\documentclass[a4paper,11pt]{amsart}
\allowdisplaybreaks
\usepackage[english]{babel}
\usepackage[utf8]{inputenc}
\usepackage[T1]{fontenc}
\usepackage{csquotes}
\usepackage[style=numeric,
useprefix,
hyperref,
backend=bibtex,
doi=false,
url=false,
isbn=false,
maxbibnames=99,
firstinits=true%,
]{biblatex}
\bibliography{./2018-09-06-BernsteinLip-BIB}

\usepackage{amssymb}
\usepackage{mathrsfs}
\usepackage{hyperref}
\usepackage[usenames,dvipsnames]{xcolor}
\hypersetup{colorlinks,%
citecolor=Black,%
filecolor=Black,%
linkcolor=Black,%
urlcolor=Black}
\usepackage{graphicx}

\usepackage{float}
\usepackage{thmtools}
\floatstyle{boxed} 
\restylefloat{figure}
\usepackage{enumitem}	% Per personalizzare gli elenchi
\newcommand{\scr}[1]{\mathscr{#1}}

\newcommand{\bb}[1]{\mathbb{#1}}
\newcommand{\cal}[1]{\mathcal{#1}}

\newcommand{\N}{\mathbb{N}}	% Numeri naturali
\newcommand{\R}{\mathbb{R}}	% Numeri reali

\newcommand{\Co}{\mathscr{C}}	% Funzioni continue
\newcommand{\ssubset}{\Subset}	% Sottoinsieme compatto

\newcommand{\dd}{\,\mathrm{d}}	% 'd' di derivata
\newcommand{\de}{\partial}		% Derivata parziale

\renewcommand{\div}{\operatorname{div}}	% div di Divergenza.

\newcommand{\grad}{\nabla}
\newcommand{\vf}{\varphi}
\newcommand{\eps}{\varepsilon}
\newcommand{\gt}{\tau}
\newcommand{\HH}{\bb H}
\newcommand{\Cow}{\mathscr{C}^1_{\mathbb{W}}}

\newcommand{\LL}{\mathcal L}
\newcommand{\LI}{{L^\infty}}
\DeclareMathOperator{\Lip}{Lip}

\newcommand{\spt}{{\rm spt}}
\usepackage{dsfont}
\newcommand{\one}{{\mathds 1\!}}	% mappa che vale 1

   \def\XXint#1#2#3{{\setbox0=\hbox{$#1{#2#3}{\int}$}
        \vcenter{\hbox{$#2#3$}}\kern-.5\wd0}}
\theoremstyle{plain}
\newtheorem{proposition}{Proposition}[section]
\newtheorem{theorem}[proposition]{Theorem}
\newtheorem{lemma}[proposition]{Lemma}
\newtheorem{corollary}[proposition]{Corollary}

\theoremstyle{definition}
\newtheorem{definition}[proposition]{Definition}
\declaretheorem[name=Remark,sibling=proposition,qed={\lower-0.3ex\hbox{$\blacklozenge$}}]{remark}

\theoremstyle{remark}

\title[Bernstein problem in Heisenberg group]{The Bernstein problem for Lipschitz intrinsic graphs in the Heisenberg group}

\subjclass[2010]{%
53C17,% sub-Riemannian geometry
49Q20% Variational problems in a geometric measure-theoretic setting
}

\keywords{%
Sub-Riemannian Geometry, %
Sub-Riemannian Perimeter, %
Heisenberg Group, %
Bernstein Problem%
}

\author[Nicolussi]{Sebastiano Nicolussi}
\address[Nicolussi]{School of Mathematics\\ University of Birmingham\\ B152TT Birmingham \\ United Kingdom}

\author[Serra Cassano]{Francesco Serra Cassano}
\address[Serra Cassano]{Dipartimento di Matematica, Università di Trento, Via Sommarive 14, 38123 Trento, Italy}
\thanks{
	Both authors have been supported by the European Unions Seventh Framework Programme, Marie Curie Actions-Initial Training Network, under grant agreement n. 607643, ``Metric Analysis For Emergent Technologies (MAnET)''.\\
	S.~N. is supported by EPSRC Grant "Sub-Elliptic Harmonic Analysis" (EP/P002447/1).\\
	F.~S.C.~is supported by MIUR, Italy, GNAMPA of INDAM and University of Trento.
}
\date{\today}

\begin{document}
\begin{abstract}
We prove that, in the %sub-Riemannian 
first Heisenberg group~$\HH$, an entire locally Lipschitz  intrinsic graph  admitting vanishing first variation of its sub-Riemannian area and non-negative second variation must be an intrinsic plane, i.e., a coset of a two dimensional subgroup of~$\HH$. Moreover  two examples are given for stressing result's sharpness.
\end{abstract}
\maketitle

\setcounter{tocdepth}{2}
\phantomsection
\addcontentsline{toc}{section}{Contents}
\tableofcontents

%%%%%%%%%%%%%%%%%%%%%%%%%%%%%%%%%%%%%%%%%%%%%%%%%%%%%%%%%%%%%%%
\section{Introduction}

%\newpage
%\subsection{Overview}
%
%%Carnot groups are homogeneous metric measure spaces
%Sub-Riemannian Carnot groups are Lie groups endowed with a left-invariant length distance that has a one-parameter group of dilations.
%
%Geometric Measure Theory on %sub-Riemannian
% Carnot groups still lacks of a regularity theory for perimeter minimizers.
%Sets with finite Riemannian perimeter have also finite sub-Riemannian perimeter.
%%On the contrary, 
%But there are sets of finite sub-Riemannian perimeter that have Riemannian Hausdorff measure strictly larger than their topological dimension, see~\cite{SC-K}.
%

% Pansu:
% P.Pansu, Geometrie du group d’Heisenberg. These pour le titre de Docteur 3`eme cycle, Universite Paris VII, (1982). 
% per area dei t-grafici
% 
% Lavoro di Serra Cassano SOME TOPICS OF GEOMETRIC MEASURE THEORY IN CARNOT GROUPS : problemi aperti
% 
% Chang Yang Wang: MR2262784
% Cheng Wang Yang 2007 Mathematician annalen
% 
% MR2481053 : dove citiamo i lavori di Giovanna, questo fa regolarità C1. 
% 
% Pauls e Garofalo (regolarità per t-grafici)
% https://arxiv.org/abs/math/0209065
% Pauls: MR2043961
% Pauls: MR2225631

Geometric Measure Theory on sub-Riemannian Carnot groups 
is a thriving research area where, despite many deep results,
fundamental questions still remain open
% Tesidipansu e MR979599
% e quello già citato
% Capogna danielli garofalo 1994, MR1312686
% MR1871966
% MR1404326
% libro di capogna
% note di francesco
% perimetri su spazi metrici ambrosio: MR1823840
% corso di ambrosio con ghezzi a parigi
% 
\cite{TesiPansu,MR979599,MR676380,MR1312686,MR1404326,MR1823840,MR1871966,MR2312336,MR3642646,MR3587666}.
%there are still fundamental open questions.
In this paper we deal with the Bernstein problem 
in the sub-Riemannian first Heisenberg group $\HH$~\cite{MR2333095,MR2455341,MR2584184,MR2472175,MR2648078,MR2609016,MR3406514}. % 
We characterize minimal entire intrinsic graphs of Lipschitz functions.
We also discuss examples with Sobolev and $C^1$-intrinsic regularity.

The Lie algebra of the Heisenberg group is spanned by three vector fields $X$, $Y$ and $Z$, 
whose only non-trivial bracket relation is $[X,Y]=Z$.
The vector fields $X$ and $Y$ are called \emph{horizontal} and they have a special role in the geometry and analysis on $\HH$.

Suitable notions of sub-Riemannian perimeter and area have been introduced on $\HH$, see \cite{TesiPansu,MR1312686,MR1404326,MR1871966} and Section~\ref{sec06152116} below for details. % MR3363670 sulla regolarità
%A theory of sub-Riemannian perimeter has been developed following the classical theory on $\R^n$.
In the theory of perimeter that has been developed,
 regular surfaces in $\HH$ play the same role as $C^1$-hyersurfaces in $\R^n$.
 A \emph{regular surface} in $\HH$ is the level set of a function $F:\HH\to\R$ with distributional derivatives $XF$ and $YF$ that are continuous and not vanishing simultaneously.
As an example of the difficulties encountered in the sub-Riemannian setting,
 we remark that
 there are regular surfaces in $\HH$ with Euclidean Hausdorff dimension strictly exceeding the topological dimension~\cite{MR2124590}.

The Bernstein problem asks to characterize area-minimizing hypersurfaces that are the graph of a function.
Two types of graphs in $\HH$ have been studied so far: \emph{$T$-graphs} and \emph{intrinsic graphs}.
The former are graphs 
%of functions $f:\R^2\to\R$ 
along the vector field $Z$ (also called $T$ in the literature): 
if $f:\R^2\to\R$, then $\Gamma_f^T=\{(x,y,f(x,y)):(x,y)\in\R^2\}$ is the $T$-graph of $f$ in coordinates.
The latter are graphs along 
%a horizontal vector field, i.e., 
a linear combination of $X$ and $Y$, which can be chosen to be $X$ up to isomorphism:
if $f:\R^2\to\R$, then $\Gamma_f=\{(0,y,t)*(f(y,t),0,0):(y,t)\in\R^2\}$ is the $X$-graph of $f$ in exponential coordinates, where $*$ denotes the group operation of $\HH$.

We say that a function, or its graph, is \emph{stationary} if the first variation of the area functional vanishes.
We call them \emph{stable} if they are stationary and the second variation of the area functional is non-negative.
See Section~\ref{sec06152116} for details in the case of intrinsic graphs.

The area functional for $T$-graphs is convex \cite{TesiPansu,MR1312686,MR2043961,MR2225631,MR2262784,MR2481053,MR3276118}.
Hence, stationary $T$-graphs are local minima.
%\cite{MR1312686,MR676380}.
Moreover, any function whose $T$-graph has finite sub-Riemannian area has (Euclidean) bounded variation \cite{MR3276118}.

%If $f\in C^2(\R^2)$, 
The Bernstein problem for $T$ graphs of functions in $C^2(\R^2)$
 has been intensively studied
 \cite{2002math......9065G,MR2165405,MR2584184,MR2609016}. %[GP, Garofalo,Pauls, 2003], 
% \cite{} % [CHMY (Cheng, Hwang, Malchiodi,Yang, 2005)], 
% [CW(Cheng, Hwang,2005 ??]. %[RR (Ritore,Rosales, 2008)].
Under this regularity assumption, a complete characterization has been given \cite{MR2584184}:
$\Gamma_f^T$ is area-minimizing if and only if there are $a$, $b$ and $c$ real such that
\begin{itemize}
\item
$f(x,y) = ax+by+c$, or
\item
$f(x,y)= xy +ax+b$ %, for some $a$, $b$ and $c$ real 
	 (up to a rotation around the $Z$-axis).
\end{itemize}
Beyond $C^2$-regularity, there are plenty of examples of minimal graphs that are not $C^2$ \cite{MR2448649}.
We also recall that there are examples of \emph{discontinuous} functions defined on a half plane whose $T$-subgraph is perimeter minimizing, see \cite[\S3.4]{MR3276118}.

The regular (but Euclidean fractal) surface constructed in \cite{MR2124590} is not a $T$-graph, but it is an intrinsic graph.
In fact, all regular surfaces are locally intrinsic graphs \cite{MR1871966}.
When the intrinsic graph of a function $f:\R^2\to\R$ is a regular surface, we\footnotemark{} write $f\in\Cow(\R^2)$ and we say that $f$ is $C^1$-intrinsic, or of class $\Cow$.
\footnotetext{In the literature, one usually writes $f:\bb W\to\bb V$, where $\bb W=\{(0,y,t):(y,t)\in\R^2\}$ and $\bb V=\{(x,0,0):x\in\R\}$). This explains the use of the letter $\bb W$.}
% the letter $\mathbb{W}$ is used  to denote the domain of $f$

An important class of intrinsic graphs are \emph{intrinsic planes}, i.e., cosets of two-dimensional Lie subgroups of $\HH$.
Their Lie algebra contains $Z$ and for this reason they are sometimes called vertical planes.
The tangents (as blow-ups at one point) of regular surfaces are intrinsic planes \cite{MR1871966}.
Intrinsic planes are area minimizers \cite{MR2333095}.

The Bernstein problem for intrinsic graphs has been also intensively studied
\cite{MR2333095,MR2455341,MR2584184,MR2472175,MR2648078,MR3406514}.
%[BASCV (Barone Adesi, S.C, Vittone, 2007)], [DGN (Danielli, Garofalo, Nhieu, 2008)] ,[MSCV (Monti, S.C., Vittone, 2008], [DGNP (Danielli, Garofalo, Nhieu, Pauls, 2009)],.......,[GR (Galli,Ritore, 2015)].
In this case, the area functional is not convex and there are stationary graphs that are not area minimizers \cite{MR2472175}.
So, any characterization of area minimizers uses both first and second variations of the area functional.
%A graph is stationary if the first variation 

%In this setting, 
The scheme of a Bernstein conjecture for intrinsic graphs is: 
 ``If $f\in\cal X$ and $\Gamma_f$ is area minimizer, then $\Gamma_f$ is an intrinsic plane'', where $\mathcal X$ is a class of functions $\R^2\to \R$.
If $\mathcal X=C^0(\R^2)\cap W^{1,1}_{loc}(\R^2)$, the conjecture is false \cite{MR2455341}.
To our knowledge, the most general positive result is for $\mathcal X=C^1(\R^2)$ in \cite{MR3406514}.
% states that if $f\in C^1(\R^2)$ and $\Gamma_f$ is area minimizer, then $\Gamma_f$ is an intrinsic plane 
%However, there are examples of functions $f\in C^0(\R^2)$ such that $\Gamma_f$ is area minimizer but it is not
We improve this result by showing that the conjecture is true for $\mathcal X=\Lip_{loc}(\R^2)$.

%We say that $f$ is \emph{stable} if the first variation of the area functional at $f$ is zero and the second variation is non-negative, see Section~\ref{sec06152116}.
%Clearly, if $\Gamma_f$ is area minimizing, then $f$ is stable.

\begin{theorem}\label{thm06110146}
	If $f\in\Lip_{loc}(\R^2)$ is stable, %a local area minimizer, 
	then  $\Gamma_f$ is an intrinsic plane.
\end{theorem}

Our proof follows the strategy of~\cite{MR2333095}:
 We will make a change of variables in the formulas for the first and second variation using so-called Lagrangian coordinates.
With this in mind, we have to show that Lagrangian coordinates exist in the first place, see Theorem~\ref{existLpLip}, and then 
%However, we need to 
take care of all regularity issues involved in the change of variables.
%In fact, 

As an intermediate step in the proof of Theorem~\ref{thm06110146}, we obtain a regularity result for stationary intrinsic graphs with Lipschitz regularity~\cite{MR2481053,MR2583494,MR2774306}.
%Notice that $\Lip_{loc}(\R^2)\not\subset\Cow(\R^2)$.
We denote by $\nabla^f$ the vector field $\de_y+f(y,t)\de_t$ on $\R^2$, see Section~\ref{sec06152116}.

\begin{theorem}\label{thm09012208}
	Let $\omega\subset\R^2$ open.
	If $f\in\Lip_{loc}(\omega)$ is stationary, then $\nabla^ff\in\Lip_{loc}(\omega)$ and $\nabla^ff$ is constant along the integral curves of $\nabla^f$.
%	 in the Lagrangian sense.
	In particular, $f\in\Cow(\omega)$.
\end{theorem}

This theorem is an extension of \cite[Theorem 1.3]{MR2583494} because we are not assuming $f$ to be a viscosity solution of the minimal surface equation, but just a distributional solution.
The example that proves Theorem~\ref{thm09022205} below will also show that 
there are distributional solutions that are not viscosity solutions in the sense of \cite[Definition 1.1]{MR2583494}, see Remark~\ref{rem09042024}.
% is shown by the example stated in the next theorem, if we restrict it to any open set $\Omega\subset\R^2$ not containing $0$ but containing $(1,0)$.\footnote{Da riscrivere questa parte}

Once the proof for the Lipschitz case is understood, 
we investigate the sharpness of Theorem \ref{thm06110146} with respect to the Lipschitz regularity of $f$ in two examples.
The first example is locally Lipschitz on $\R^2$ except for one point, it is stable but $\Gamma_f$ is not an intrinsic plane.
See Figure~\ref{fig09030019} at page~\pageref{fig09030019} for a picture of $\Gamma_f$.
%In fact, we give an example of a function $f$ which is stable, continuous, Lipschitz continuous in $\R^2\setminus\{0\}$ and with Sobolev regularity on the whole $\R^2$ but $\Gamma_f$ is not an intrinsic plane.

\begin{theorem}\label{thm09022205}
	There is $f\in Lip_{\rm loc}(\R^2\setminus\{0\})\cap W^{1,p}_{\rm loc}(\R^2)$ with $1\le\,p<\,3$ that is stable, but $\Gamma_f$ is not an intrinsic plane.
\end{theorem}

The second example fails to be Lipschitz on a Cantor set, but it is $C^1$-intrinsic.
See Figure~\ref{fig09030022} at page~\pageref{fig09030022} for a picture of $\Gamma_f$.

\begin{theorem}\label{thm06102321}
	There is $f\in W^{1,2}_{loc}(\R^2)\cap\Cow(\R^2)\cap\Lip_{loc}(\R^2\setminus(\{0\}\times C))$, where $C\subset[0,1]$ is the Cantor set,
	that is stable,
%	 such that $I_f(\psi)=0$ and $II_f(\psi)\ge0$ for all $\psi\in\Co^\infty_c(\R^2)$, 
	but $\Gamma_f$ is not an intrinsic plane.
\end{theorem}

For both examples of Theorem~\ref{thm09022205} and~\ref{thm06102321}, we don't know whether their intrinsic graphs are area minimizing.

We conclude by recalling
% we recall 
 some open problems in geometric measure theory on $\HH$ and higher Heisenberg groups.
 First, the Bernstein conjecture with $\mathcal X=\Cow(\R^2)$ is still open.
 Second, a regularity theorem for perimeter minimizers %in $\HH$ 
 is still missing \cite{MR3363670,MR3388885,MR3682744,MR3397002}.
 Third, 
 if we don't assume that the intrinsic graph has locally finite Euclidean area, then the variational formulas we used are not valid anymore and known alternative variations haven't found useful applications yet \cite{MR3558526,MR3753176,MR3276118}.

\subsection*{Plan of the paper}
In Section~\ref{sec06152116} we present a few preliminaries notions and notations.
In Section~\ref{sec09022311}, we prove that Lagrangian parametrizations exist for locally Lipschitz functions.
Section~\ref{sec09022316} is devoted to the characterization of stationary locally Lipschitz intrinsic graphs and  the proof of Theorem~\ref{thm09012208} is presented.
Section~\ref{sec09022317} concerns the consequences of stability, and thus the proof of Theorem~\ref{thm06110146}.
In Section~\ref{sec09022318} we study a class of stationary surfaces, called graphical strips, with low regularity.
Sections~\ref{sec11241414} and~\ref{sec09011226} are devoted to the examples of Theorems~\ref{thm09022205} and~\ref{thm06102321}, respectively.

%%%%%%%%%%%%%%%%%%%%%%%%%%%%%%%%%%%%%%%%%%%%%%%%%%%%%%%%%%%%%%%
\subsection*{Acknowledgements}
This paper benefited from
%The authors thank M.~Ritoré for 
 fruitful discussions
 with M.~Ritoré: The authors want to thank him.

%%%%%%%%%%%%%%%%%%%%%%%%%%%%%%%%%%%%%%%%%%%%%%%%%%%%%%%%%%%%%%%
\section{Preliminaries and notation}\label{sec06152116}
The Heisenberg group $\HH$ is represented in this paper as $\R^3$ endowed with the group operation
\[
(x,y,z)(x',y',z') = \left(x+x',y+y',z+z'+\frac12(xy'-x'y) \right) .
\]
In this coordinates, an orthonormal frame of the horizontal distribution is
\[
X = \de_x - \frac{y}{2} \de_z,
\qquad
Y = \de_y + \frac{x}{2} \de_z .
\]
The \emph{sub-Riemannian perimeter} of a measurable set $G\subset\HH$ in an open set $\Omega\subset\HH$ is
\[
P_{sR}(G;\Omega) = \sup\left\{ \int_G (X\psi_1+Y\psi_2) \dd L^3 : \psi_1,\psi_2\in C^\infty_c(\Omega),\ \psi_1^2+\psi_2^2\le 1 \right\} ,
\]
where $X\psi_1+Y\psi_2$ is the divergence of the vector field $\psi_1X+\psi_2Y$.
A set $G$ is a \emph{perimeter minimizer} in $\Omega\subset\HH$ if for every $F\subset\HH$ measurable with $(G\setminus F)\cup (F\setminus G)\ssubset \Omega$, it holds $P_{sR}(G;\Omega) \le P_{sR}(F;\Omega)$.
A set $G$ is a \emph{locally perimeter minimizer} in $\Omega\subset\HH$ if every $p\in\Omega$ has a neighborhood $\Omega'\subset\Omega$ such that $G$ is perimeter minimizer in $\Omega'$.

Given a function $f:\omega\to\R$, $\omega\subset\R^2$, its intrinsic graph $\Gamma_f\subset\HH$ is the set of points
\[
(0,y,t)(f(y,t),0,0) = \left( f(y,t) , y , t-\frac12 y f(y,t) \right),
\]
for $(y,t)\in\omega$.
The \emph{intrinsic gradient} of $f$ is $\nabla^ff$, where $\nabla^f=\de_y+f\de_t$ is a vector field on $\omega$.
%\[
%\nabla^f=\de_y+f\de_t .
%\]
The function $\nabla^ff:\omega\to\R$ is well defined when $f\in W^{1,1}_{loc}(\omega)$.
We say that $f$ is \emph{$C^1$-intrinsic}, or $f\in\Cow(\omega)$, if $f\in C^0(\omega)$ and $\grad^ff\in C^0(\omega)$.
We say that $f$ is a \emph{weak Lagrangian solution of $\Delta^ff=0$} on $\omega$ if for every $p\in\omega$ there is at least one integral curve of $\grad^f$ passing through $p$ along which $\grad^ff$ is constant.
See~\cite{MR3753176} for further discussion about this definition.

The \emph{graph area functional} is defined, for every $E\subset \omega$ measurable, by
\[
\scr A_f(E) := \int_E \sqrt{1+(\nabla^ff)^2} \dd\cal L^2 .
\]
Such area functional descends from the perimeter measure of the graph, that is, $\scr A_f(E) = P_{sR}(G_f\cap (E\cdot\R))$, where $G_f:=\{(0,y,t)\cdot(\xi,0,0):\xi\le f(y,t)\}$ is the subgraph of $f$, and $E\cdot\R=\{(0,y,t)\cdot(\xi,0,0):\xi\in\R,\ (y,t)\in E\}$.
A function $f\in W^{1,1}_{loc}(\omega)$ is \emph{(locally) area minimizing} if $G_f$ is (locally) perimeter minimizing in $\omega\cdot\R$.

%We say that $f\in W^{1,1}_{loc}(\R^2)\cap C^0(\R^2)$ is a \emph{local area minimizer} if for all $E\subset\R^2$ and all $g\in W^{1,1}_{loc}(\R^2)\cap C^0(\R^2)$ with $\spt(f-g)\ssubset E$ we have
%\[
%\scr A_f(E)\le \scr A_g(E) .
%\]
%This notion of area minimizer is a priori weaker than the subgraph $G_f$ being locally perimeter minimizer: Indeed, it is not known whether minimizing the area among graphs does imply minimizing the perimeter among sets of finite perimeter, with the same boundary conditions.

We say that $f\in W^{1,1}_{loc}(\omega)$ is \emph{stationary} if for all $\varphi\in C^\infty_c(\omega)$
\[
I_f(\varphi) := \left.\frac{\dd}{\dd\epsilon}\scr A_{f+\epsilon\varphi}(\spt(\varphi))\right|_{\epsilon=0} = 0 .
\]
We say that $f\in W^{1,1}_{loc}(\omega)$ is \emph{stable} if it is stationary and for all $\varphi\in C^\infty_c(\omega)$
\[
II_f(\varphi) := \left.\frac{\dd^2}{\dd\epsilon^2}\scr A_{f+\epsilon\varphi}(\spt(\varphi))\right|_{\epsilon=0} \ge 0 .
\]
The functionals $I_f$ and $II_f$ are called \emph{first} and \emph{second variation} of $f$, respectively.
It is clear that, if $f$ is a local area minimizer, then it is stable.

By~\cite[Remark 3.9]{MR2455341}, if $f\in W^{1,1}_{loc}(\omega)$, for some $\omega\subset\R^2$ open, then
\begin{align*}
	I_f(\varphi) &= - \int_{\omega} \frac{\grad^ff}{\sqrt{1+(\grad^ff)^2}} (\grad^f\varphi+\de_tf\,\varphi) \dd\LL^2 , \\
	II_f(\varphi) &=\int_{\omega} 
			\left[ \frac{(\grad^f\varphi+\de_tf\,\varphi)^2}{\left( 1+(\grad^ff)^2 \right)^{3/2}}
			+  \frac{\grad^ff}{\sqrt{1+(\grad^ff)^2}} \de_t(\varphi^2)\right]
		\dd\LL^2
\end{align*}
for all $\varphi\in C^\infty_c(\omega)$.
Notice that the formal adjoint of $\grad^f$ is $(\grad^f)^*\varphi=-\grad^f\varphi-\de_tf\,\varphi$.
By means of the triangle and the Hölder inequalities, one can easily show the following lemma.

\begin{lemma}\label{lem08261641}
	Let $\omega\subset\R^2$ open, then $I_f(\varphi)$ and $II_f(\varphi)$ are continuous in the $W^{1,2}_{loc}$ topology for $f\in W^{1,2}_{loc}(\omega)$ or $\varphi\in W^{1,2}_{0}(\omega)$ fixed, that is,
	\begin{itemize}
	\item[(i)] If $f_n\to f$ in $W^{1,2}_{loc}(\omega)$ and $\varphi\in W^{1,2}_{0}(\omega)$ with $\spt(\varphi)\Subset\omega$,
	then $\lim_{n\to\infty} I_{f_n}(\varphi)=I_f(\varphi)$  and $\lim_{n\to\infty} II_{f_n}(\varphi)=II_f(\varphi)$.
	\item[(ii)] If $f\in W^{1,2}_{loc}(\omega)$, $\varphi_n\to \varphi$ in $W^{1,2}_{loc}(\omega)$ and there exists $\omega'\ssubset\omega$ with $\varphi_n\in C^\infty_c(\omega')$ for each $n$, then $\lim_{n\to\infty} I_{f}(\varphi_n)=I_f(\varphi)$  and $\lim_{n\to\infty} II_{f}(\varphi_n)=II_f(\varphi)$.
	\end{itemize}

\end{lemma}

%%%%%%%%%%%%%%%%%%%%%%%%%%%%%%%%%%%%%%%%%%%%%%%%%%%%%%%%%%%%%%%
%%%%%%%%%%%%%%%%%%%%%%%%%%%%%%%%%%%%%%%%%%%%%%%%%%%%%%%%%%%%%%%
\section[Lagrangian parametrizations]{Existence and regularity of Lagrangian homeomorphisms}\label{sec09022311}

%%%%%%%%%%%%%%%%%%%%%%%%%%%%%%%%%%%%%%%%%%%%%%%%%%%%%%%%%%%%%%%
\subsection{Definition of Lagrangian parametrization}   
Roughly speaking, a Lagrangian parametrization of $\nabla^f$ is a continuous  ordered selection of integral curves of the vector field $\nabla^f$ on $\omega$ with respect to a parameter $\tau$, which covers all of $\omega$.
For $\omega\subset\R^2$ and $r\in\R$, we set
\[
\omega_{1,r}:=\left\{y\in\R:\, (y,r)\in \omega\right\}
\qquad\text{ and }\qquad
\omega_{2,r}:=\left\{t\in\R:\, (r,t)\in \omega\right\} .
\]

\begin{definition}[Lagrangian parameterization]\label{Lagrpar}
	Let $\omega\subset\R^2$ be an open set and $f:\omega\to\R$ a continuous function.
	A \emph{ Lagrangian parameterization} associated with the vector field $\nabla^f$
	is a  continuous map $\Psi:\,\tilde\omega\to\omega$, with $\tilde\omega$ open, that satisfies
	\begin{itemize}
		\item[(L.1):] $\Psi(\tilde\omega)=\,\omega$;
		\item[(L.2):] $\Psi(s,\tau)= (s,\chi(s,\tau))$ for a suitable continuous function $\chi:\,\tilde\omega\to\R$ and, for every $s\in\R$, the function
		$\tilde\omega_{2,s}\ni \tau\mapsto\chi(s,\tau)$ is nondecreasing;
		\item[(L.3):] for every $ \tau\in\R$, for every $(s_1,s_2)\subset\tilde\omega_{1,\tau}$, the curve $(s_1,s_2)\ni s\mapsto\Psi(s,\tau)$ is absolutely continuous and it is an integral curve of $\nabla^f$, that is
		\begin{equation*}%\label{E:generalODE}
			\partial_s\Psi(s,\tau)=\nabla^f(\Psi(s,\tau))\quad{\rm a.e.\ }s\in(s_1,s_2).
		\end{equation*}
	\end{itemize}
	Equivalently, condition (L.3) can be rephrased as: 
	for every $ \tau\in\R$, for every $(s_1,s_2)\subset\tilde\omega_{1,\tau}$, we have $\de_s\chi(s,\tau)=f(s,\chi(s,\tau))$, for almost every $s\in(s_1,s_2)$.

	A Lagrangian parameterization $\Psi:\,\tilde\omega\to\omega$, is said to be {\it absolutely continuous} if   it satisfies the  Lusin (N) condition, that is, for every $E\subset\tilde\omega$, if $\cal L^2(E)=\,0$ then $\cal L^2(\Psi(E))=\,0$.
	A {\it Lagrangian homeomorphism} $\Psi:\,\tilde\omega\to\omega$, is an injective Lagrangian parameterization.
	By the Invariance of Domain Theorem, the injectivity implies that a Lagrangian homeomorphism is indeed a homeomorphism.
\end{definition}

\begin{remark} 
	Definition \ref{Lagrpar} is an equivalent version of the definition of {\it Lagrangian parameterization to   function $f:\,\omega\to\R$}, introduced in   \cite{MR3400438} and then extended in \cite{MR3537830}, for studying different notions of continuous weak solutions for balance laws. 
\end{remark}

\begin{remark}\label{Ncondequiv} 
	Observe that, by Fubini's theorem, a Lagrangian parameterization $\Psi:\,\tilde\omega\to\omega$, $\Psi(s,\tau)= (s,\chi(s,\tau))$ (associated with a vector field $\nabla^f$) is absolutely continuous if and only if for each $\mathcal L^2$-negligible set $E\subset\tilde\omega$, we have that
\begin{equation*}%\label{Ncondfiber}
\mathcal L^1(\chi(s,E_{2,s}))=\,0\quad\mathcal L^1\text{-a.e.~}s\in\R\,. \qedhere
\end{equation*}
\end{remark}
\begin{remark}\label{rem05122131}
	Lagrangian parametrizations are not unique: If $\Psi(s,\tau)=(s,\chi(s,\tau))$ is a (absolutely continuous) Lagrangian parametrization and $\rho:\R\to\R$ is an absolutely continuous homeomorphism with $\rho'>0$, then $(s,\tau)\mapsto (s,\chi(s,\rho(\tau)))$ is again a (absolutely continuous) Lagrangian pa\-ra\-me\-tri\-za\-tion.
\end{remark}

%%%%%%%%%%%%%%%%%%%%%%%%%%%%%%%%%%%%%%%%%%%%%%%%%%%%%%%%%%%%%%%
\subsection{Rules for the change of variables}
A relevant feature of an absolutely continuous Lagrangian parameterization associated with the vector field $\nabla^f$ is that we can use it for a change of variables.
This is the essential tool of the Lagrangian approach to the equation of minimal surfaces equation. 

When a homeomorphism $\Psi:\tilde\omega\to\omega$ is fixed, we will denote by $\tilde u$ or $(u)^{\widetilde{ }}$ the composition $u\circ\Psi:\tilde\omega\to\bar\R$ with a function $u:\omega\to\bar\R$.

One can  prove the following area formula for absolutely continuous Lagrangian parameterizations:
\begin{lemma}[Area formula for absolutely continuous Lagrangian parameterizations]\label{lemafLp} 
	Let $\Psi:\tilde\omega\to\omega$, $\Psi(s,\gt)=\,(s,\chi(s,\gt))$, be an absolutely continuous Lagrangian parameterisation associated with a vector field $\nabla^f$. 
	Let $\eta:\omega\to\bar\R$ be a Borel summable function. 
	Then
	\[
	\int_{\tilde\omega}\tilde\eta(s,\tau) \,\de_\gt\chi(s,\tau)\,\dd s\dd\gt
		=\,\int_\omega\eta(y,t)\,\dd y\dd t\,.
	\]
\end{lemma}
\begin{proof}
Let us begin to observe that
\begin{equation}\label{chiinW11}
\chi\in W^{1,1}_{\rm loc}(\tilde\omega)\,,
\end{equation}
and 
\begin{equation}\label{dertauchinonneg}
\partial_\gt\chi(s,\gt)\ge\,0\quad \mathcal L^2\text{-a.e.~}(s,\gt)\in\tilde\omega\,.
\end{equation}
Indeed, by Definiton \ref{Lagrpar} (L.3), it follows that, for each $\gt\in\R$, for every $(s_1,s_2)\subset\tilde\omega_{1,\tau}$,  
\begin{equation}\label{chi_tauac}
(s_1,s_2)\ni s\mapsto\chi(s,\tau)\text{  is absolutely continuous.}
\end{equation}
On the other hand, by Remark \ref{Ncondequiv}, it follows that for a.e.~$s\in\R$, for every $(\gt_1,\gt_2)\subset\tilde\omega_{2,s}$,  $(\gt_1,\gt_2)\ni \tau\mapsto\chi(s,\tau)$ satisfies the Lusin (N) condition, being also continuous and non decreasing, we can also infer (see, for instance, \cite[Theorem 7.45]{gariepy1995modern} or \cite{MR3729481})
	%\footnote{Non capisco la referenza: Non trovo Garapie--Ziemer su Mathscinet. Il libro presumo sia
	%\url{https://www.amazon.co.uk/Real-Analysis-Ronald-F-Gariepy/dp/0534944043}
	%oppure \url{https://books.google.co.uk/books/about/Modern_Real_Analysis.html?id=CBQNNQAACAAJ&redir_esc=y}.}
\begin{equation}\label{chi_sac}
(\gt_1,\gt_2)\ni \tau\mapsto\chi(s,\tau)\text{  is absolutely continuous and non decreasing.}
\end{equation}
By \eqref{chi_tauac} and \eqref{chi_sac} and applying a well-known result about Sobolev spaces (see \cite[\S 4.9.2]{MR3409135}), \eqref{chiinW11} and \eqref{dertauchinonneg} follow. 
By \eqref{chiinW11}  and since $\Psi$ satisfies the Lusin (N)-condition, we can  the area formula for Sobolev mappings (see, for instance,
\cite[Theorem A.35]{MR3184742}), that is
\begin{equation}\label{areaformula}
\int_{\tilde\omega}\eta(\Psi(s,\gt))\,|J_\Psi(s,\gt)|\,dsd\gt=\,\int_{\Psi(\tilde\omega)}\eta(y,t)\,N(\Psi,\tilde\omega,(y,t))\,dydt
\end{equation}
where the multiplicity function $N(\Psi,\tilde\omega,(y,t))$ of $\Psi$  is defined as the number of preimages of $(y,t)$ under $\Psi$ in $\tilde\omega$ and 
\[
\begin{split}
J_\Psi(s,\gt)&:=\,\det D\Psi(s,\gt)=\det \left [
\begin{matrix}
1&0\\
\de_s\chi(s,\gt)&\displaystyle{{\de_\gt\chi(s,\gt)}}
\end{matrix}
\right]\\
&=\,\partial_\gt\chi(s,\gt)\quad \mathcal L^2\text{-a.e.~}(s,\gt)\in\tilde\omega\,.
\end{split}\,
\]
The left-hand side of \eqref{areaformula} is thus $\int_{\tilde\omega}\tilde\eta(s,\tau) \,\de_\gt\chi(s,\tau)\,\dd s\dd\gt$.

Let us show that $N=1$ for almost every $(y,t)\in\omega$.
First, observe that,
\begin{equation*}%\label{charmultipl}
N(\Psi,\tilde\omega,(y,t))=\,N(\chi(y,\cdot),\tilde\omega_{2,y},t)\quad\forall\,(y,t)\in\R^2\,.
\end{equation*}
Second, if $y\in\R$, then the set $\left\{t\in\R:\, N(\chi(y,\cdot),\tilde\omega_{2,y},t)\ge\,2\right\}$
is at most countable, because $\tau\mapsto\chi(y,\tau)$ is continuous and non-decreasing.
We conclude that $N=1$ for almost every $(y,t)\in\omega$ as claimed.\\
Therefore, the right hand side of \eqref{areaformula} is $\int_{\Psi(\tilde\omega)}\eta(y,t) dydt$.
\end{proof}

However, in order to perform the change of variables also on derivatives, we need additional assumptions on $\Psi$.
For our purposes, we will consider the case when $\Psi$ is locally biLipschitz.
\begin{remark} \label{chainlip}Let $g\in Lip_{\rm loc}(\omega)$ and $\Psi:\,\tilde\omega\to\omega$ be a locally biLipschitz homeomorphism. Then it is easy to see that the following chain rule holds:
\begin{equation}\label{CR}
\tilde g\in Lip_{\rm loc}(\tilde\omega)\text{ and } D\tilde g(s,\gt)=\,Dg(\Psi(s,\gt))\,D\Psi(s,\gt)\quad\mathcal L^2\text{a.e.-}(s,\gt)\in\tilde\omega\,,
\end{equation}
where $D\tilde g$ and $Dg$ respectively denote the gradient of $\tilde g$ and $g$ understood as a  $1\times 2$ matrix  and $D\Psi$ denotes the Jacobian $2\times 2$ matrix of $\Psi$.  Indeed it is trivial that $ \tilde g\in Lip_{\rm loc}(\tilde\omega)$ being the composition of Lipschitz functions. Thus, by Radamecher's theorem, there exist $D\tilde g$, $Dg$ and $D\Psi$ either from the pointwise point of view and in sense of distribution on their domain. Moreover, since both $\Psi$ and $\Psi^{-1}$ satisfy the Lusin (N) condition,  if $\omega_g$ and $\tilde\omega_\Psi$ respectively denote the points of differentiability of $g$ in $\omega$ and of $\Psi$ in $\tilde\omega$, then 
\begin{equation*}
\mathcal L^2(\omega\setminus (\Psi(\tilde\omega_\Psi)\cap\omega_g))= \mathcal L^2(\tilde\omega\setminus (\tilde\omega_\Psi\cap\Psi^{-1}(\omega_g))=\,0\,.
\end{equation*}
Thus, for each $(s,\gt)\in\tilde\omega_\Psi\cap\Psi^{-1}(\omega_g)$, $\tilde g$ is differentiable in classical sense at  $(s,\gt)$ and \eqref{CR} holds.
\end{remark}

\begin{theorem}[Rules for the change of variables]\label{thm05112336super}
	Let $\Psi:\tilde\omega\to\omega$, $\Psi(s,\gt)=\,(s,\chi(s,\gt))$, be a locally biLipschitz Lagrangian homeomorphism associated with a vector field $\nabla^f$ and assume that $f\in \Lip(\omega)$.
	Then we have
	\begin{equation}\label{eq05121214}
	\de_s\de_\tau\chi = \de_\tau\de_s\chi = (\de_tf)^{\widetilde{}}\, \de_\tau\chi = \de_\tau\tilde f ,
	\end{equation}
	and for every compact $K\subset\tilde\omega$ there is $C>0$ such that $\de_\tau\chi(s,\tau)>C$ for almost all $(s,\tau)\in K$.
	Furthermore, for each $\varphi\in Lip(\omega)$,
	\begin{align}\label{eq08261555}
	(\de_t\varphi)^{\widetilde{}} &= \frac{\de_\tau\tilde\varphi}{\de_\tau\chi} , &
	(\de_y\varphi)^{\widetilde{}} &= \de_s\tilde\varphi - \frac{\tilde f}{\de_\tau\chi} \de_\tau\tilde\varphi \qquad\text{ and }  &
	(\grad^f\varphi)^{\widetilde{}} &= \de_s\tilde\varphi .
	\end{align}
\end{theorem}
\begin{proof}
	The first equality in~\eqref{eq05121214} has to be considered as an equality of distributions, being $\de_s$ and $\de_\tau$ distributional derivations.
	Next, if $f\in \Lip(\omega)$, then we are allowed to differentiate with respect to $\tau$ the identities
	\begin{equation*}%\label{}
	\partial_s\chi(s,\gt) = f(s,\chi(s,\gt)) = \tilde f(s,\tau) \quad\mathcal L^2{-a.e.~}(s,\gt)\in\tilde\omega .
	\end{equation*}
	Thus we obtain the other two identities in~\eqref{eq05121214}.
	
	The Jacobian matrix of $\Psi^{-1}$ at $\Psi(s,\tau)$ is 
	\[
	D\Psi^{-1}(\Psi(s,\gt))=\frac1{\de_\tau\chi(s,\tau)}
		\begin{pmatrix} \de_\tau\chi(s,\tau) & 0 \\ -\de_s\chi(s,\tau) & 1 \end{pmatrix}.
	\]
	Since $\Psi$ is biLipschitz, the determinant of this matrix is locally bounded from above, hence $\de_\tau\chi$ is locally bounded away from zero.
	
	Finally, the equalities in~\eqref{eq08261555} follow directly from Remark \ref{chainlip}.
\end{proof}

%%%%%%%%%%%%%%%%%%%%%%%%%%%%%%%%%%%%%%%%%%%%%%%%%%%%%%%%%%%%%%%
\subsection{Existence of biLipschitz Lagrangian homeomorphisms}

\begin{theorem}[Existence of a biLipschitz Lagrangian homeomorphism associated with a Lipschitz vector field $\nabla^f$]\label{existLpLip} 
	Let $\omega\subset\R^2$ be an open set and $f\in\Lip(\omega)\cap L^\infty(\omega)$.
	Then there exists a locally biLipschitz Lagrangian homeomorphism $\Psi:\tilde\omega\to\omega$, $\Psi(s,\tau)=(s,\chi(s,\tau))$, associated with $\nabla^f$.
	Moreover, if $\omega=\R^2$, then $\tilde\omega=\R^2$ and such Lagrangian parametrization $\Psi$ is unique if we require $\chi(0,\tau)=\tau$ for all $\tau\in\R$.
\end{theorem}
\begin{proof}
	We can assume that $\omega=\R^2$.
	Indeed, by McShane's Extension Theorem of Lipschitz functions (see \cite{MR2039660}), if $f\in\Lip(\omega)\cap L^\infty(\omega)$, then there is an extension $f^*\in\Lip(\R^2)\cap L^\infty(\R^2)$ with $\Lip(f^*)=\Lip(f)$ and $\|f^*\|_{L^\infty(\R^2)} = \|f\|_{L^\infty(\omega)}$.
	Moreover, if $\Phi^*:\R^2\to\R^2$ is a Lagrangian homeomorphism associated with $\nabla^{f^*}$ with the properties stated in the Theorem, then its restriction $\Phi:=\Phi^*|_{\tilde\omega}$ to $\tilde\omega=\Phi^{-1}(\omega)$ still have all the stated properties.
	So, we assume  $\omega=\R^2$.
	
	Since $f\in\Lip(\R^2)$ and it is bounded, and by standard results from ODE's Theory (see \cite{MR1929104}), it is well-known that for every $(s_1,\tau_1)\in\R^2$ there is a unique $C^1$ function $\gamma:\R\to\R$ such that
	\begin{equation}\label{eq06220926}
	\begin{cases}
	\gamma'(s)=f(s,\gamma(s)) \qquad\text{for a.e.~}s\in\R, \\
	\gamma(s_1)=\tau_1 .
	\end{cases}
	\end{equation}
	Such $\gamma$ is in fact of class $C^{1,1}$.
	For $s_1,s_2,\tau_1\in\R$, define
	\[
	\cal X\big( s_1,\tau_1 ; s_2 \big) := \gamma(s_2) ,
	\]
	where $\gamma$ is the solution of the system above, depending on the initial conditions $(s_1,\tau_1)$.
	Using Grönwall's lemma, one can easily prove that, for every $s_1,s_2\in\R$ and every $\tau,\tau'\in\R$ we have
	$|\cal X\big( s_1,\tau  ; s_2 \big)-\cal X\big( s_1,\tau'  ; s_2 \big)| \le |\tau-\tau'| \exp(L|s_2-s_1|)$.
	Hence, the map $\tau\mapsto\cal X\big( s_1,\tau  ; s_2 \big)$ is locally Lipschitz, with a Lipschitz constant that is locally uniform in $s_1$ and $s_2$.
	% DIMOSTRAZIONE DEL FATTO DI SOPRA: DA QUI
	%	\footnote{
	%	Let $L=\Lip(f)$;
	%	Let $\gamma(s)=\cal X\big( s_1,\tau  ; s \big)$ and $\eta(s)=\cal X\big( s_1,\tau'  ; s \big)$;
	%	Set $u(s)=|\gamma(s)-\eta(s)|$;
	%	Then $u'(s)\le|\gamma'(s)-\eta'(s)|=|f(s,\gamma(s))-f(s,\eta(s))|\le L|\gamma(s)-\eta(s)| = Lu(s)$;
	%	By Grönwall's lemma, we have $u(s_2) \le u(s_1) \exp(|\int_{s_1}^{s_2} L \dd v|) = u(s_1)\exp(L|s_2-s_1|)$;
	%	Therefore, $|\cal X\big( s_1,\tau  ; s_2 \big)-\cal X\big( s_1,\tau'  ; s_2 \big)| \le |\tau-\tau'| \exp(L|s_2-s_1|)$.
	%	}
	% A QUI.
	
	By the uniqueness of solutions, for all $s_1,s_2,s_3,\tau_1\in\R$ the following identity holds:
	\[
	\cal X\big( s_2 , \cal X\big( s_1,\tau_1 ; s_2\big) ; s_3 \big) = \cal X\big( s_1,\tau_1 ; s_3 \big) .
	\]
	In particular, the map $\tau\mapsto\cal X\big( s_2,\tau  ; s_1 \big)$ is the inverse of $\tau\mapsto\cal X\big( s_1,\tau  ; s_2 \big)$, and thus they are locally biLipschitz homeomorphisms.
	
	Define $\chi:\R^2\to\R$,
	\[
	\chi(s,\tau) := \cal X\big( 0,\tau ; s \big) .
	\]
	By the previous discussion, $\tau\mapsto\chi(s,\tau)$ is a locally biLipschitz homeomorphism $\R\to\R$, for all $s\in\R$.
	Since $|f|$ is bounded, then, for all $s,s',\tau\in\R$,
	\begin{equation*}%\label{eq05121540}
	|\chi(s,\tau)-\chi(s',\tau)| \le \|f\|_{L^\infty} |s-s'| .
	\end{equation*}
	So, since $\chi$ is locally Lipschitz in $s$ and in $\tau$ with uniform constants, then $\chi:\R^2\to\R$ is locally Lipschitz.
	
	Define $\Psi:\R^2\to\R^2$ as $\Psi(s,\tau)=(s,\chi(s,\tau))$, which is locally Lipschitz.
	Notice that, by the uniqueness of solution to the above ODE, $\Psi$ is injective.
	Moreover, by the existence of a global solution to the above ODE for every initial conditions, $\Psi$ is surjective.
	By the Invariance of Domain Theorem, $\Psi$ is a homeomorphism.
	
	Moreover, $\Psi$ is locally biLipschitz.
	Indeed, its inverse is $\Psi^{-1}(y,t)=(y,\rho(y,t))$ with 
	\[
	\rho(y,t) = \cal X\big( y,t  ; 0 \big) .
	\]
	As before, we can prove that $\rho$ is locally Lipschitz in each variable independently. 
	Indeed, on the one hand we already showed that $\rho$ is locally Lipschitz in $t$, with the Lipschitz constant that is locally uniform in $y$.
	On the other hand, we have
	\begin{align*}
	|\rho(y,t)-\rho(y',t)|
	&= \left| \cal X\big( y',\cal X\big(y,t;y'\big)  ; 0\big) - \cal X\big( y',t ;0\big) \right| \\
	&\le C \left| \cal X\big( y,t ;y' \big) - t \right| \\
	&\le C \|f\|_{L^\infty} |y-y'| .
	\end{align*}
	We conclude that $\Psi:\R^2\to\R^2$ is a locally biLipschitz Lagrangian homeomorphism.
	
	Finally, notice that $\chi(0,\tau)=\tau$ for all $\tau\in\R$ and that the uniqueness of such $\chi$ follows from the uniqueness of solutions to~\eqref{eq06220926}.
\end{proof}

%%%%%%%%%%%%%%%%%%%%%%%%%%%%%%%%%%%%%%%%%%%%%%%%%%%%%%%%%%%%%%%
%%%%%%%%%%%%%%%%%%%%%%%%%%%%%%%%%%%%%%%%%%%%%%%%%%%%%%%%%%%%%%%
\section{Consequences of the first variation}\label{sec09022316}

If $f\in\Lip_{loc}(\omega)$ is a local area minimizer, then the  first variation formula vanishes, i.e., see~\cite{MR2455341}:
\begin{equation}\label{eq06221238}\tag{$1^{st}$VF}
	\forall \varphi\in C^\infty_c(\omega)
	\qquad\qquad
	I_f(\varphi)=0 .
\end{equation}
By Lemma~\ref{lem08261641}, the condition~\eqref{eq06221238} can be extended to $\varphi\in\Lip_c(\omega)$.

The aim of this section is to prove the following theorem.

\begin{theorem}\label{thm05111605}
	Suppose that $f\in\Lip_{loc}(\Omega)$ satisfies~\eqref{eq06221238}, where $\Omega\subset\R^2$ is open.
	Then $\grad^ff$ is locally Lipschitz, thus $f\in\Cow(\Omega)$, and $f$ is a weak Lagrangian solution of $\Delta^ff=0$ on $\Omega$.

	More in details, let $\omega\ssubset\Omega$, so that $f\in\Lip(\omega)\cap L^\infty(\omega)$,
	and let $\Psi:\tilde\omega\to\omega$, $\Psi(s,\tau)=(s,\chi(s,\tau))$, be a locally biLipschitz Lagrangian homeomorphism associated with $\grad^f$.
	Such a function exists by Theorem~\ref{existLpLip}.
	Let $s_1,s_2,\tau_1,\tau_2\in\R\cup\{+\infty,-\infty\}$ be such that $(s_1,s_2)\times(\tau_1,\tau_2)\subset\tilde\omega$ and let $\hat s\in(s_1,s_2)$.
	Then, for all $(s,\tau)\in(s_1,s_2)\times(\tau_1,\tau_2)$, 
	\begin{equation}\label{eq05122152}
	\chi(s,\tau) = a(\tau)\frac{(s-\hat s)^2}{2} + b(\tau) (s-\hat s) + c(\tau)
	\end{equation}
	where $c:(\tau_1,\tau_2)\to\R$ is locally biLipschitz on its image, $a(\tau)=\grad^ff(\hat s,c(\tau))$ and $b(\tau)=f(\hat s,c(\tau))$.
	Moreover, both $a$ and $b$ are locally Lipschitz.
	
	Up to an further locally biLipschitz change of variables, one can also assume $c(\tau)=\tau$ for all $\tau\in(\tau_1,\tau_2)$.
\end{theorem}

The proof is postponed after a lemma, which highlights a crucial step, that is, the change of variables in the integral~\eqref{eq06221238} via a Lagrangian homeomorphism for $\grad^f$.
Once we can make this step, the conclusion follows quite directly.

\begin{lemma}\label{lem05122135}
	Suppose that $f\in\Lip_{loc}(\omega)\cap L^\infty(\omega)$ satisfies~\eqref{eq06221238} on $\omega\subset\R^2$ open.
	Let $\Psi:\tilde\omega\to\omega$, $\Psi(s,\tau)=(s,\chi(s,\tau))$, be a locally biLipschitz Lagrangian homeomorphism associated with $\grad^f$.
	Then
	\begin{equation}\label{1stvfhcchi}
	\int_{\tilde\omega} \frac{\de_s^2\chi}{\sqrt{1+(\de_s^2\chi)^2}}\, \de_s\theta \, \dd\mathcal L^2 = 0\qquad\text{ for all }\theta\in \Lip_c(\tilde\omega).
	\end{equation}
\end{lemma}
\begin{proof}	
	By Lemma~\ref{lemafLp} and Theorem~\ref{thm05112336super}, 
	we can perform the change of variables $(y,t)=\Psi(s,\tau)=(s,\chi(s,\tau))$ in \eqref{eq06221238} to obtain
	\begin{equation}\label{chvarminsurfeq}
	\begin{split}
	0&=\int_{\omega} \frac{\grad^ff}{\sqrt{1+(\grad^ff)^2}} (\grad^f\varphi+\de_tf\,\varphi) \,d\mathcal L^2\\
	&=\int_{\tilde\omega} \frac{\partial_s^2\chi}{\sqrt{1+(\partial_s^2\chi)^2}}\left( \partial_s\tilde\vf\,\partial_\tau\chi+\partial_\tau\tilde f\,\tilde\vf\right) \,d\mathcal L^2 \\
	\end{split}
	\end{equation}
	for each $\varphi \in Lip_c(\omega)$. 
	
	Fix $\theta\in\Lip_c(\tilde\omega)$. 
	We would like to substitute $\tilde\varphi$ with $\frac{\theta}{\de_\tau\chi}$ in~\eqref{chvarminsurfeq}, but $\de_\tau\chi$ does not need to be Lipschitz.  
	Let $K:=\,\spt(\theta)$ and, if $\eps>\,0$, let
	\[
	K_\eps:=\left\{(s,\gt)\in\R^2:\,{\rm dist}((s,\gt),K)<\,\eps\right\}\,.
	\]
	Then $(K_\eps)_\eps$ is a family of bounded open sets containing $K$ and there exists $\eps_0>\,0$ such that $K_{\eps_0}\Subset\tilde\omega$. 
	Since $\Psi$ is locally biLipschitz, there are $C>c>0$ with 
	\begin{equation}\label{dertaichipos}
	C>\de_\tau\chi>c \text{ a.e.~in }  K_{\eps_0}\,,
	\end{equation}
	we can successfully apply an argument by smooth approximation.

	Let $\rho_\epsilon\in C^\infty_c(\R^2)$ be a family of mollifiers and define $\chi_\epsilon:=\chi*\rho_\epsilon \in C^\infty(K_\epsilon)$.
	By \eqref{dertaichipos} and the properties of convolution with mollifiers, we have the following facts for $\eps\in (0,\eps_0/2)$:
	\begin{enumerate}[label=(\roman*)]
	\item
	$c\le\de_\gt\chi_\epsilon\le C$ on $K$ and $\chi_\epsilon\to\chi$ a.e.~on $K$;
	\item
	$\grad\chi_\epsilon\to\grad\chi$ a.e.~ on $K$ and
	$\|\grad\chi_\epsilon\|_{L^\infty(K)} \le \|\grad\chi\|_{L^\infty(K_{\eps_0})}$;
	\item
	$\de_s\de_\tau\chi_\epsilon = (\de_s\de_\tau\chi)*\rho_\epsilon = (\de_\tau\tilde f)*\rho_\epsilon$, therefore 
	$\de_s\de_\tau\chi_\epsilon\to\de_s\de_\tau\chi$ a.e.~on $K$ and
	$\|\de_s\de_\tau\chi_\epsilon\|_{L^\infty(K)} \le \|\de_\tau\tilde f\|_{L^\infty(K_{\eps_0})}$.
	\end{enumerate}
	
	For every $\epsilon>0$ small enough, the function $(s,\tau)\mapsto \frac{\theta(s,\tau)}{\de_\tau\chi_\epsilon(s,\tau)}$ is well defined and belongs to $\Lip_c(\tilde\omega)$. 
	Since $\Psi$ is locally biLipschitz, there exists $\varphi_\epsilon\in\Lip(\omega)$ such that $\tilde\varphi_\epsilon=\frac{\theta}{\de_\tau\chi_\epsilon}$.
	Moreover, we have
	\[
	\partial_s\tilde\varphi_\epsilon \, \de_\tau\chi + \de_\tau \tilde f \, \tilde\varphi_\epsilon
	= \de_s\theta \frac{\de_\tau\chi}{\de_\tau\chi_\epsilon}
		+ \theta \frac{-\de_s\de_\tau\chi_\epsilon\,\de_\tau\chi + \de_\tau\de_s\chi\,\de_\tau\chi_\epsilon}{(\de_\tau\chi_\epsilon)^2}
	\]
	Since $\de_s^2\chi=\de_s\tilde f\in L^\infty_{loc}(\tilde\omega)$, then $ \frac{\partial_s^2\chi}{\sqrt{1+(\partial_s^2\chi)^2}} \in L^\infty_{loc}(\tilde\omega)$. 
	From the facts (i)--(iii) above and the Lebesgue Dominated Convergence Theorem, we obtain
	\begin{align*}
	0 &= \lim_{\epsilon\to0} \int_{\tilde\omega} \frac{\partial_s^2\chi}{\sqrt{1+(\partial_s^2\chi)^2}}\left( \partial_s\tilde\vf_\epsilon\,\partial_\tau\chi+\partial_\tau\tilde f\,\tilde\vf_\epsilon\right) \,d\mathcal L^2 \\
	&= \int_{\tilde\omega} \frac{\partial_s^2\chi}{\sqrt{1+(\partial_s^2\chi)^2}} \de_s\theta
	\,d\mathcal L^2 .
	\end{align*}
	We have so proven~\eqref{1stvfhcchi}.
\end{proof}

\begin{proof}[Proof of Theorem~\ref{thm05111605}]
	By Lemma~\ref{lem05122135}, $\chi$ satisfies~\eqref{1stvfhcchi}.
	Therefore, $\de_s^2\chi$ is a constant function in $s$, that is, for almost every $\tau\in(\tau_1,\tau_2)$ the function $s\mapsto\chi(s,\tau)$ is a polynomial of degree two.
	Thus, there are measurable functions $a,b,c:(\tau_1,\tau_2)\to\R$ such that~\eqref{eq05122152} holds.
	
	First, notice that $c(\tau) = \chi(\hat s,\tau)$ for a.e.~$\tau\in(\tau_1,\tau_2)$.
	Therefore, the map $c$ is a locally biLipschitz homeomorphism from $(\tau_1,\tau_2)$ onto its image in $\R$, with $c'>0$ almost everywhere.
	
	Second, since $f(\hat s,\chi(\hat s,\tau)) = \de_s\chi(\hat s,\tau) = b(\tau)$, the function $b$ is in fact locally Lipschitz.
	
	Third, if $\delta>0$ is such that $\hat s+\delta<s_2$, then we have
	\[
	\chi(\hat s+\delta,\tau) = a(\tau) \delta^2 + b(\tau) \delta + c(\tau)
	\]
	for a.e.~$\tau\in(\tau_1,\tau_2)$, and thus the function $a$ is also locally Lipschitz.
	Moreover, from Theorem~\ref{thm05112336super} we have $\grad^ff(s,\chi(s,\tau)) = \de_s\tilde f(s,\tau) = \de_s^2\chi(s,\tau)$ for a.e.~$(s,\tau)\in(s_1,s_2)\times(\tau_1,\tau_2)$.
	Since $\de_s^2\chi(s,\tau)=a(\tau)$, then we obtain $a(\tau)=\grad^ff(s,\chi(s,\tau))$ for a.e.~$(s,\tau)\in(s_1,s_2)\times(\tau_1,\tau_2)$.
	
	Finally, notice that $\grad^ff(y,t)=a(\tau(\Psi^{-1}(y,t)))$ is locally Lipschitz on $\Phi((s_1,s_2)\times(\tau_1,\tau_2))$.
\end{proof}

%NUOVO TESTO AGGIUNTO:

After Theorem~\ref{thm05111605}, we can improve the existence result of Theorem~\ref{existLpLip} for $f\in\Lip_{loc}(\R^2)$ that satisfies~\eqref{eq06221238}.

\begin{corollary}\label{cor09102030}
	Suppose that $f\in\Lip_{loc}(\R^2)$ satisfies~\eqref{eq06221238}.
	Then there exists a unique locally biLipschitz Lagrangian homeomorphism $\Psi:\R^2\to\R^2$, $\Psi(s,\tau)=(s,\chi(s,\tau))$, for $f$ such that $\chi(0,\tau)=\tau$ for all $\tau\in\R$.
	Moreover, $\chi$ is of the form
	\[
	\chi(s,\tau) = a(\tau) \frac{s^2}{2} + b(\tau) s + \tau ,
	\]
	where $a,b:\R\to\R$ are the locally Lipschitz functions $a(\tau)=\grad^ff(0,\tau)$ and $b(\tau)=f(0,\tau)$.
\end{corollary}
\begin{proof}
	By Theorem~\ref{thm05111605} and by the Invariance of Domain Theorem, the function $\Psi$ in the corollary is a locally biLipschitz Lagrangian homeomorphism $\Psi:\R^2\to\Psi(\R^2)$.
	Again by Theorem~\ref{thm05111605}, $f$ belongs to $\Cow(\R^2)$ and it is a weak Lagrangian solution of $\Delta^ff=0$.
	Therefore, we can apply \cite[Lemma~3.5]{MR3753176} and obtain that $\Psi$ is indeed surjective.
\end{proof}
%\begin{proof}
%	For every $R>0$, let $\omega_R:=(-R,R)^2\subset\R^2$.
%	Then we have $f_R:=f|_{\omega_R}\in\Lip(\omega_R)\cap L^\infty(\omega_R)$.
%	By Theorem~\ref{thm05111605}, for every $R>0$ there is $\Psi_R:\tilde\omega_R\to\omega_R$, $\Psi_R(s,\tau)=(s,\chi_R(s,\tau))$, locally biLipschitz Lagrangian homeomorphism such that
%	\[
%	\chi_R(s,\tau) = a_R(\tau)\frac{s^2}{2} + b_R(\tau) s + c_R(\tau) .
%	\]
%	Up to composing $\Psi_R$ with $(s,\tau)\mapsto (s,c_R^{-1}(\tau))$, which is biLipschitz, we take $c_R(\tau)=\tau$.
%	Notice that $(0,0)\in int(\tilde\omega_R)$.
%	By the uniqueness of solutions to the ODE~\eqref{eq06220926}, for every $R,R'>0$ the maps $\Psi_R$ and $\Psi_{R'}$ coincides on $\tilde\omega_R\cap\tilde\omega_{R'}$.
%	Therefore, if $R'>R$, then $\tilde\omega_{R}\subset\tilde\omega_{R'}$ and $\Psi_{R}$ is the restriction of $\Psi_{R'}$ on $\tilde\omega_{R}$.
%	
%	So, let $\tilde\omega:=\bigcup_{R>0}\tilde\omega_R$ and $\Psi:\tilde\omega\to\R^2$ be the surjective Lagrangian homeomorphism with $\Psi|_{\tilde\omega_R}=\Psi_R$.
%	We claim that $\tilde\omega=\R^2$.
%	Indeed, on the one hand, we clearly have $\{0\}\times\R\subset\tilde\omega$.
%	On the other hand, for every $\tau\in\R$, if $\tau< R$, then the curve $s\mapsto(s,\chi_R(s,\tau))$ is indeed defined for all $s\in\R$.
%\end{proof}
\begin{remark}
	We want to stress that, in Corollary~\ref{cor09102030}, the condition~\eqref{eq06221238} is crucial.
	For instance, consider $f(y,t)=t^2$, which is locally Lipschitz on $\R^2$ but does not satisfy~\eqref{eq06221238}.
	The maximal integral curves of $\grad^f=\de_y+t^2\de_t$ are not defined on the whole line $\R$.
	Indeed, $\gamma(s)=(s,\frac{\tau_1}{1-\tau_1 (s-s_1)})$ is the solution to~\eqref{eq06220926} with such $f$ and it is not defined at $s=\frac{1+\tau_1s}{\tau_1}$.
\end{remark}

%%%%%%%%%%%%%%%%%%%%%%%%%%%%%%%%%%%%%%%%%%%%%%%%%%%%%%%%%%%%%%%
%%%%%%%%%%%%%%%%%%%%%%%%%%%%%%%%%%%%%%%%%%%%%%%%%%%%%%%%%%%%%%%
\section{Consequences of the second variation}\label{sec09022317}

If $f\in\Lip_{loc}(\omega)$ is a local area minimizer, then  the  second variation formula is non-negative,i.e., see~\cite{MR2455341}:
\begin{equation}\label{eq06221239b}\tag{$2^{nd}$VF}
	\forall \varphi\in C^\infty_c(\omega)
	\qquad\qquad
	II_f(\varphi)\ge\,0 .
\end{equation}
By Lemma~\ref{lem08261641}, the condition~\eqref{eq06221239b} can be extended to $\varphi\in\Lip_c(\omega)$.

We recall that there are plenty of examples of functions $f\in\Lip_{loc}(\omega)$, for suitable open sets $\omega$, that satisfy both conditions~\eqref{eq06221238} and~\eqref{eq06221239b}, as we wil see in Proposition \ref{prop06061521}, see also  \cite{MR2472175,MR2448649}.

%In fact, any function of class $C^2$ satisfying~\eqref{eq06221238} also satisfies~\eqref{eq06221239b} outside the set (referring to~\eqref{eq05122152})
%\[
%\{(s,\chi(s,\tau)): a(\tau) (s-\hat s) + b(\tau)=0 \} ,
%\]
%that is, outside the curve drawn by the vertices of the parables that rule the plane, 
%Indeed, it is possible to construct a local calibration for the graph $\Gamma_f$ outside this curve.

The aim of this section is to prove the following theorem, which is a restatement of Theorem~\ref{thm06110146}.

\begin{theorem}\label{thm05111656}
	Suppose that $f\in\Lip_{loc}(\R^2)$ satisfies~\eqref{eq06221238} and~\eqref{eq06221239b}.
	Then $\grad^ff$ is constant and thus the graph $\Gamma_f$ of $f$ is an intrinsic plane.
	
	More precisely, let $\Psi(s,\tau)=(s,\chi(s,\tau))$ be the only Lagrangian pa\-ra\-me\-tri\-za\-tion associated with $\grad^f$ such that $\chi(0,\tau)=\tau$ for all $\tau$, which exists by Corollary~\ref{cor09102030}.
	Then
	\[
	\chi(s,\tau) = a \frac{s^2}{2} + b s + \tau
	\]
	with $a,b\in\R$.
\end{theorem}

We postpone the proof after a number of lemmas.
The overall strategy is the same as in~\cite{MR2333095}. On the other hand let us point out that we are not allowed to carry out the same calculations as in ~\cite{MR2333095} in computing the second variation formula. In fact, here function $f$ is supposed to be only locally Lipschitz continuous and not $C^2$. Thus we have to adapt the previous calculations. 

\begin{lemma}\label{lem04061138}
	Let $a,b\in\Lip_{loc}(\R)$, and define
	\[
	\chi(s,\tau) = \frac{a(\tau)}{2} s^2 + b(\tau) s + \tau .
	\]
	Assume that $\Psi:(s,\tau)\mapsto(s,\chi(s,\tau))$ is a Lagrangian parametrization for  $f\in\Lip_{loc}(\R^2)$.
	Then:
	\begin{enumerate}
	\item
	For all $\tau_1,\tau_2\in\R$, either $a(\tau_1)=a(\tau_2)$ and $b(\tau_1) = b(\tau_2)$,
 	or $2 \big(a(\tau_1)-a(\tau_2)\big) (\tau_1-\tau_2) > \big(b(\tau_1)-b(\tau_2)\big)^2 $;
	\item
	For almost every $\tau\in\R$ we have either $a'(\tau)= b'(\tau)=0$, or $2a'(\tau)> b'(\tau)^2$.
	\end{enumerate}
\end{lemma}
\begin{proof}
	First of all, notice that, by the uniqueness of solutions to \eqref{eq06220926} for $f$ locally Lipschitz, the Lagrangian parametrization $\Psi$ here is the one constructed in Theorem~\ref{existLpLip}.
	In particular, this $\Psi$ is a locally biLipschitz homeomorphism.
	
	The first part of the lemma is contained in Lemma~3.2 of \cite{MR3753176}.
	Before proving the second part, notice that $2a'(\tau) \ge b'(\tau)^2$ follows directly from the inequality $2 \big(a(\tau_1)-a(\tau_2)\big) (\tau_1-\tau_2) \ge \big(b(\tau_1)-b(\tau_2)\big)^2 $, which holds for every $\tau_1,\tau_2\in\R$.
	Moreover, since $\Psi$ is locally biLipschitz, the function $f\circ\Psi$ is differentiable for almost every $(s,\tau)\in\R^2$.
	
	In order to show the second part of the lemma, we show that the sets
	\[
	E_k=\left\{ \tau\in\R : 
		\begin{array}{l} 
		f\circ\Psi\text{ is differentiable at $(s,\tau)$ for a.e.~}s , \\
		a,b\text{ are differentiable at $\tau$,}\\
		k^{-2} <2a'(\tau)=b'(\tau)^2 \text{ and }|\tau|\le2k
		\end{array}
	\right\} 
	\]
	have zero measure, for all $k\in\N$. Assume, by contradiction that $\cal L^1(E_k)>\,0$ for a given $k$.
	Notice that, if $\tau\in E_k$, then for almost every $s\in\R$
	\[
	\de_t f(s,\chi(s,\tau)) = \frac{ a'(\tau)s+b'(\tau) }{a'(\tau)s^2/2 + b'(\tau) s + 1} 
		=\frac1{sb'(\tau)/2+1} .
	\]
	The denominator of this expression vanishes at $s=-\frac{2}{b'(\tau)}\in[2k,2k]$.
	Therefore, for every $N\in\N$ and for every $\tau\in E_k$ there is $I_{k,N,\tau}\subset[-2k,2k]$ with $\cal L^1(I_{k,N,\tau})>0$, such that $|\de_t f(s,\chi(s,\tau))|\ge N$ for all $s\in I_{k,N,\tau}$. Let $B_{N,k}:=\bigcup_{\tau\in E_k} I_{k,N,\tau}\times\{\tau\}\subset[-2k,2k]^2$.  $B_{N,k}$ need not to be $\cal L^2$-measurable. However, since $\cal L^2$ is a  Borel outer measure, there exists a Borel set $B^*_{N,k}\subset [-2k,2k]^2$ 
	such that
	\[
	B_{N,k}\subset B^*_{N,k} \text{ and }\cal L^2(B^*_{N,k})=\,\cal L^2(B_{N,k})\,.
	\] 
	Let
	\[
	B^*_{N,k,\tau}:=\left\{s\in\R:\,(s,\tau)\in B^*_{N,k}\right\}\text{ if }\tau\in\R\,,
	\]
	then
	\[
	B^*_{N,k,\tau}\supset I_{k,N,\tau}\text{ for each }\tau\in E_k\,.
	\]
	By Fubini's theorem,
	\[
	\cal L^2(B^*_{N,k})=\,\int_\R\cal L^1(B^*_{N,k,\tau})\,d\tau\ge\,\int_{E_k}\cal L^1(B^*_{N,k,\tau})\,d\tau>\,0
	\]
	It follows that, for every $N\in\N$, 
	\[
	\mathrm{ess}\sup_{\Psi(B^*_{N,k})}|\de_t f|\ge N\,,
	\] 
	where  $\Psi(B^*_{N,k})\subset\R^2 $ is a Borel set  of $\cal L^2$-positive measure.
	This is a contradiction, because $f\in \Lip([-2k,2k]^2)$.
	
	We conclude that $\LL^1(E_k)=0$ for all $k\in\N$ and thus that (2) holds.
\end{proof}

\begin{lemma}\label{lem04061131}
	Let $a,b\in\Lip_{loc}(\R)$, and define
	\[
	\chi(s,\tau) = \frac{a(\tau)}{2} s^2 + b(\tau) s + \tau .
	\]
	Assume that $\Psi:(s,\tau)\mapsto(s,\chi(s,\tau))$ is a locally biLipschitz Lagrangian homeomorphism for  $f\in\Lip_{loc}(\R^2)$ that satisfies \eqref{eq06221239b}.
	Then, for all $\tilde\varphi\in\Lip_c(\R^2)$,
	\begin{equation*}%\label{eq06071139}
	\int_{\R^2} (\de_s\tilde\varphi)^2 \frac{a's^2/2 + b's + 1}{(1+a^2)^{3/2}} 
	- \tilde\varphi^2 \frac{2a'-b'^2}{(a's^2/2 + b's + 1)(1+a^2)^{3/2}}  
		\dd s\dd \tau
	\ge 0 ,
	\end{equation*}
	where $a$, $a'$ and $b'$ are functions of $\tau$, while $\tilde\varphi$ is a function of $(s,\tau)$.
\end{lemma}
\begin{proof}
	Since the map $\Psi$ is a locally biLipschitz homeomorphism,
	given $\tilde\varphi\in\Lip_c(\R^2)$ we have $\varphi:=\tilde\varphi\circ\Psi^{-1}\in\Lip_c(\R^2)$ and $II_f(\varphi)\ge0$.
	Performing a change of variables via $\Psi$ using Theorem~\ref{thm05112336super}, we have:
	\begin{multline*}
	\!\!\!\!II_f(\varphi)\!
	=\!\!\! \int_{\R^2} \!\! \left(\frac{
		\left(\de_s\tilde\varphi+\tilde\varphi\frac{a's+b'}{a's^2/2 + b's + 1}\right)^2 }{(1+a^2)^{3/2}} + 
		\frac{
			\frac{\de_\tau(\tilde\varphi^2)}{a's^2/2 + b's + 1} a
			}{(1+a^2)^{1/2}}
		\right)\!\!
		(a's^2/2 + b's + 1)
		\dd s\text{d}\tau \\
	= \int_{\R^2} \bigg[ (\de_s\tilde\varphi)^2 \frac{a's^2/2 + b's + 1}{(1+a^2)^{3/2}}  
		+ \tilde\varphi^2 \frac{(a's+b')^2}{(a's^2/2 + b's + 1)(1+a^2)^{3/2}}  +\hfill\\ \hfill
		+ \de_s(\tilde\varphi^2) \frac{a's+b'}{(1+a^2)^{3/2}}
		+ \de_\tau(\tilde\varphi^2)\frac{a}{(1+a^2)^{1/2}} \bigg]\dd s\dd \tau \\
	= \int_{\R^2} \bigg[ (\de_s\tilde\varphi)^2 \frac{a's^2/2 + b's + 1}{(1+a^2)^{3/2}}  
		+ \tilde\varphi^2 \frac{(a's+b')^2}{(a's^2/2 + b's + 1)(1+a^2)^{3/2}}  + \hfill\\\hfill
		- \tilde\varphi^2 \frac{a'}{(1+a^2)^{3/2}}
		- \tilde\varphi^2 \frac{a'}{(1+a^2)^{3/2}} \bigg] \dd s\dd \tau \\
	= \int_{\R^2} \bigg[ (\de_s\tilde\varphi)^2 \frac{a's^2/2 + b's + 1}{(1+a^2)^{3/2}}  
		+ \tilde\varphi^2 \frac{b'^2-2a'}{(a's^2/2 + b's + 1)(1+a^2)^{3/2}}  
		\bigg]\dd s\dd \tau . \hfill
	\end{multline*}
\end{proof}

%%%%%%%%%%%%%%%%%%%%%%%%%%%%%%%%%%%%%%%%%%%%%%%%%%%%%%%%%%%%%%%
The following lemma is proven in \cite[p.45]{MR2333095}.
\begin{lemma}\label{lem06070850}
	Let $A,B\in\R$ be such that $B^2\le2A$ and set $h(t):=At^2/2 + Bt + 1$.
	If
	\begin{equation*}%\label{eq06061805}
	\int_{\R} \phi'(t)^2 h(t) \dd t \ge (2A-B^2) \int_{\R} \phi(t)^2 \frac1{h(t)} \dd t
	\qquad
	\forall \phi\in C^1_c(\R),
	\end{equation*}
	then $B^2=2A$.
\end{lemma}

\begin{proof}[Proof of Theorem~\ref{thm05111656}]
	By Corollary~\ref{cor09102030}, 
	 $\chi(s,\tau)=a(\tau)\,s^2/2+b(\tau)\,s+\tau$ for some $a,b\in\Lip_{loc}(\R)$.
	By Lemma~\ref{lem04061131}, we have, for all $\tilde\varphi\in\Lip_c(\R^2)$,
	\[
	\int_{\R^2} (\de_s\tilde\varphi)^2 \frac{a's^2/2 + b's + 1}{(1+a^2)^{3/2}} 
	- \tilde\varphi^2 \frac{2a'-b'^2}{(a's^2/2 + b's + 1)(1+a^2)^{3/2}}  
		\dd s\dd \tau
	\ge 0 .
	\]
	By standard arguments (taking for example $\varphi(x,y):=\varphi_1(x)\varphi_2(y)$) we can infer that for almost every $\tau\in\R$ and all $\tilde\varphi\in\Lip_c(\R)$
	\[
%	\frac1{(1+a^2)^{3/2}} 
	\int_\R \tilde\varphi'(s)^2 (a's^2/2 + b's + 1) \dd s
	\ge
%	\frac1{(1+a^2)^{3/2}} 
	\int_\R \tilde\varphi^2 \frac{2a'-b'^2}{(a's^2/2 + b's + 1)}  \dd s ,
	\]
	where $a$, $a'$ and $b'$ are functions of $\tau$.
	By Lemma~\ref{lem04061138}, we have $b'^2\le2a'$ for almost every $\tau\in\R$.
	Thus, we can apply Lemma~\ref{lem06070850} and obtain that $b'(\tau)^2=2a'(\tau)$ for almost every $\tau\in\R$.
	By Lemma~\ref{lem04061138} again, we obtain $b'(\tau)=a'(\tau)=0$ for almost every $\tau\in\R$.
\end{proof}

%%%%%%%%%%%%%%%%%%%%%%%%%%%%%%%%%%%%%%%%%%%%%%%%%%%%%%%%%%%%%%%
%%%%%%%%%%%%%%%%%%%%%%%%%%%%%%%%%%%%%%%%%%%%%%%%%%%%%%%%%%%%%%%
\section{\texorpdfstring{$\Cow$}{C1W}-graphical strips}\label{sec09022318}
In this section we will study the functions appearing in Theorem \ref{thm05111605} with $b\equiv\,0$ and $\hat s=\,0$. 
Their intrinsic graph  has been called \emph{graphical strip} in \cite{MR2472175}, where they have been studied under $C^2$ regularity. % for~$f$.
This type of surface in $\HH$ has the shape of a helicoid: it contains the vertical axis $\{x=y=0\}$ and the intersection with $\{(x,y,z):(x,y)\in\R^2\}$ is a line for every $z\in\R$. Here we will study the case when $f$ could be less regular than $C^2$.

\begin{proposition}\label{prop06061521}
	If $a:\R\to\R$ is continuous and non-decreasing, then the map $(s,\tau)\mapsto(s,a(\tau)\frac{s^2}{2}+\tau)$ is a homeomorphism $\R^2\to\R^2$ and there is exactly one function $f\in\Cow(\R^2)$ such that for all $s\in\R$ and all $\tau\in\R$:
	\begin{equation}\label{eq06110006}
	f\left(s,a(\tau)\frac{s^2}{2}+\tau\right)=a(\tau)s .
	\end{equation}
	The function $f$ has the following properties:
	\begin{enumerate}[label=(\roman*)]
	\item 	$\grad^ff(s,a(\tau)\frac{s^2}{2}+\tau) = a(\tau)$;
%	\item 	$\grad^f\grad^ff=0$: in particular, for all $\in\Co^\infty_c(\R^2)$, the first variation of the area is zero, i.e., $I_f(\psi)=0$;
	\item 	$f$ is locally Lipschitz on $\R^2\setminus\{y=0\}$,
		and if $a\in\Lip_{loc}(\R)$, then $f$ is locally Lipschitz on $\R^2$;
	\item 	if $a\in\Lip_{loc}(\R)$, then~\eqref{eq06221238} holds; % $I_f(\psi)=0$ for all $\psi\in\Co^\infty_c(\R^2)$;
	\item 	if $a\in\Lip_{loc}(\R)$, then 
	\[
	II_{f}(\varphi) 
	= \int_{\R^2}
		(\de_s\tilde\varphi)^2 \frac{(\frac{a'}2s^2+1)}{(1+a^2)^{3/2}} 
		-2 \tilde\varphi^2 \frac{ a' }{ (1+a^2)^{3/2} (\frac{a'}2s^2+1) }
	\dd s\dd\tau ,
	\]
	where $a$ and $a'$ are functions in $\tau$ and $\tilde\varphi(s,\tau) := \varphi(s,a(\tau)\frac{s^2}{2}+\tau)$.
	\end{enumerate}
\end{proposition}
\begin{proof}
	By \cite[Lemma~3.3]{MR3753176}, the map $(s,\tau)\mapsto(s,a(\tau)\frac{s^2}{2}+\tau)$ is a homeomorphism $\R^2\to\R^2$.
	By \cite[Remark~3.4]{MR3753176}, there is a unique function $f\in\Cow$ such that~\eqref{eq06110006} and $(i)$ hold.
	
	Next, we show $(ii)$. % that $f$ is locally Lipschitz on $\{( y, t): y\neq0\}$.
%	First, fix $ y>k$ and $ t'> t$.
%	
%	Let $\tau'>\tau$ be such that $ t=a(\tau)\frac{ y^2}{2}+\tau$ and similarly for $ t'$.
%	If $a(\tau')=a(\tau)$, then $f( y, t)=f( y, t')$.
%	If $a(\tau')>a(\tau)$, 
	Let $ y, y', t, t', t'',\tau,\tau'\in\R$ be such that $ y \cdot y'>0$ and
	\begin{align*}
	( y, t) &=  ( y,\frac12a(\tau) y^2+\tau) , \\
	( y', t') &=  ( y',\frac12a(\tau') y'^2+\tau') , \\
	( y, t'') &= ( y,\frac12a(\tau') y^2+\tau') .
	\end{align*}
	Observe first that, if $t''\neq t$, then $\tau'\neq\tau$. Thus, if $a(\tau')=\,a(\tau)$, we can infer that
	\[
	f(y,t'')-f(y,t)=\,(a(\tau')-a(\tau))y=\,0\,.
	\]
	Otherwise
	\begin{equation}\label{eq07121641}
	\begin{split}
	\left|\frac{f( y, t'')-f( y, t)}{ t''- t}\right|
	&= \left|\frac{(a(\tau')-a(\tau))  y}{\frac12 (a(\tau')-a(\tau)) y^2 + (\tau'-\tau) } \right| \\ 
	&= \left| \frac{1}{ \frac y2 + \frac1 y \frac{\tau'-\tau}{a(\tau')-a(\tau)}} \right|
	\le \frac2{| y|} ,
	\end{split}
	\end{equation}
	because $\frac{\tau'-\tau}{a(\tau')-a(\tau)}>\,0$.
	Second, we estimate
	\begin{equation}\label{eq07121641b}
	\begin{split}
	|f( y', t')-f( y, t)|
	&\le |f( y', t')-f( y, t'')|+|f( y, t'')-f( y, t)| \\
	&\le  |a(\tau')| | y'- y| + \frac{2}{| y|} | t''- t| \\
	&\le |a(\tau')| | y'- y| + \frac{2}{| y|} (| t''- t'|+| t'- t|) \\
	&\le |a(\tau')| | y'- y| + \frac{1}{| y|} |a(\tau')| | y^2- y'^2| + \frac{2}{| y|} | t'- t| \\
	&= |a(\tau')| \left(1+ \frac{ | y+ y'|}{| y|}  \right) | y'- y| + \frac{2}{| y|} | t'- t| .
	\end{split}
	\end{equation}
	This shows that $f$ is locally Lipschitz on $\{( y, t): y\neq0\}$.
	If $a$ is locally Lipschitz and $I\subset\R$ is a bounded interval, then for $\tau,\tau'\in I$ we have $\frac{\tau'-\tau}{a(\tau')-a(\tau)}\ge \frac1L$ for some $L>0$ depending on $I$.
	Thus, we obtain in~\eqref{eq07121641}
	\[
	\left|\frac{f( y, t'')-f( y, t)}{ t''- t}\right| \le L|y| .
	\]
	The estimate~\eqref{eq07121641b} is then
	\[
	|f( y', t')-f( y, t)|
	\le |a(\tau')| \left(1+ L\frac{ | y+ y'|}{2} |y| \right) | y'- y| + L|y| | t'- t| ,
	\]
	which shows that $f\in\Lip_{loc}(\R^2)$.

	Let's prove $(iii)$. %I want to give a proof of the third point.
	Assume that $a$ is locally Lipschitz.
	From $(ii)$ we know that $f$ is locally Lipschitz, and thus the Lagrangian parametrization $\Psi(s,\tau)=(s,\chi(s,\tau))$ with $\chi(s,\tau)=a(\tau)s^2/2+\tau$ is a biLipschitz homeomorphism by Theorem~\ref{existLpLip}.
	Performing a change of variables as in Lemma~\ref{lemafLp} and Theorem~\ref{thm05112336super}, we obtain for all $\varphi\in C^\infty_c(\R^2)$
	\begin{align*}
	I_f(\varphi) 
	&= \int_{\R^2} \frac{a(\tau)}{\sqrt{1+a(\tau)^2}} 
		\left( \de_s\tilde\varphi(s,\tau) +  \frac{ \tilde\varphi(s,\tau) a'(\tau)s }{ a'(\tau)s^2/2+1 } \right)
		(a'(\tau)s^2/2+1) \dd s\dd\tau \\
	&= \int_{\R^2} \frac{a(\tau)}{\sqrt{1+a(\tau)^2}} 
		\frac{\de}{\de s}\left( \tilde\varphi(s,\tau) \, (a'(\tau) s^2/2 + 1)  \right)
		 \dd s\dd\tau
	= 0 .
	\end{align*}

	Finally, part $(iv)$ has already been proven in Lemma~\ref{lem04061131}, because by $(ii)$ the function $f$ is locally Lipschitz when $a\in\Lip_{loc}(\R)$.
\end{proof}
\begin{remark}
	Notice that, if there exists $\tau\in\R$ such that $\lim_{\tau'\to\tau}\frac{a(\tau')-a(\tau)}{\tau'-\tau}=\infty$, then, from~\eqref{eq07121641}, we get, for each $y\neq 0$, $\lim_{t''\to t}\left|\frac{f( y, t'')-f( y, t)}{ t'- t}\right| = \frac{2}{|y|}$, and therefore $f$ is not locally Lipschitz on $\R^2$. An example of such phenomenon is the one in Section~\ref{sec09011226}.
\end{remark}

%The intrinsic graph of a $f:\R^2\setminus\{x=y=0\}\to\R$ 

In our coordinates $(x,y,z)$ for $\HH$, the intrinsic graph of functions as in Proposition~\ref{prop06061521} have the shape of helicoids:
%The intrinsic graph of $f$,
\[
\Gamma_f = \{(0,0,\tau)+s(a(\tau),1,0) : (s,\tau)\in\R^2\} .
\]
%is a local area minimizer outside the vertical axis, i.e., for every $p\in\Gamma_f\setminus\{x=y=0\}$ there is $\Omega\subset\bb H$ open such that $\Gamma_f$ is area minimizer in $\Omega$.
%In the following proof we will give an explicit parametrization.
Moreover, we have the following result for the horizontal vector field $\Omega:=\bb H\setminus\{x=y=0\}$
\begin{equation}\label{eq07121716}
\nu(x,y,z) := -\frac{y}{\sqrt{x^2+y^2}} X|_{(x,y,z)} + \frac{x}{\sqrt{x^2+y^2}} Y|_{(x,y,z)} .
\end{equation}
%By direct computation one easily shows that $\div(\nu)=0$.

\begin{proposition}\label{prop06211845}
	The vector field $\nu$ is divergence free in $\Omega=\bb H\setminus\{x=y=0\}$,
%	 and it coincides with the unit normal to the intrinsic graph $\Gamma_f$ for any $f$ as in Proposition~\ref{prop06061521}.
%	In particular, 
	and it is a local calibration for the intrinsic graph $\Gamma_f$ for any $f$ as in Proposition~\ref{prop06061521}.
%	 all  graphs outside the vertical axis $\{x=y=0\}$.
	
	As  a consequence, $\Gamma_f$ is a local area minimizer outside the vertical axis, i.e., for every $p\in\Gamma_f\setminus\{x=y=0\}$ there is $U\subset\bb H$ open such that $\Gamma_f$ is area minimizer in $U$.
\end{proposition}
\begin{proof}%[Proof of Proposition~\ref{prop06211845}]
	It is clear that the distributional divergence of $\nu$ in $\Omega$ is 
	\[
	\div\nu = - X\left(\frac{y}{\sqrt{x^2+y^2}}\right) + Y\left(\frac{x}{\sqrt{x^2+y^2}}\right)
	= 0 .
	\]
	Next, let $G_f$ be the subgraph of $f$, i.e., $G_f=\{(0,y,t)(\xi,0,0):\xi\le f(y,t)\}$.
	It is well known (see~\cite[Theorem~1.2]{MR2223801}) that $G_f$ is a set of locally finite perimeter and that its reduced boundary is the intrinsic graph $\Gamma_f$.
	We describe $\Gamma_f$ as  image of the map $G:\R^2\to\R^3$, 
	$G(s,\tau)=(0,s,\chi(s,\tau))(f(s,\chi(s,\tau)) , 0,0)$.
	Since $\chi(s,\tau)=a(\tau)s^2/2+\tau$, then $G(s,\tau) = (0,0,\tau)+s(a(\tau),1,0)$.
%%	so that the image of $G$ is $G(\R^2)= \Gamma_f$.
	Hence, its unit normal is
	\[
	\nu_{G_f}(G(s,\tau)) = -\frac{1}{\sqrt{1+a(\tau)^2}} X|_{G(s,\tau)} + \frac{a(\tau)}{\sqrt{1+a(\tau)^2}} Y|_{G(s,\tau)} .
	% \frac{1}{\sqrt{1+(\grad^ff)^2}} X + \frac{\grad^ff}{\sqrt{1+(\grad^ff)^2}} Y .
	\]
	By a direct computation, one easily shows that $\nu_{G_f}(G(s,\tau)) = \nu(G(s,\tau))$.
%
%	.\\
%	Next, suppose that the function $a$ in~\eqref{eq06110006} is smooth and	
%	define $G(s,\tau):=(0,s,\chi(s,\tau))(f(s,\chi(s,\tau)) , 0,0)$, where $\chi(s,\tau)=a(\zeta)t^2/2+\zeta$, so that the image of $G$ is $G(\R^2)= \Gamma_f$.
%	Then, one easily computes that $G(s,\tau) = (0,0,\tau)+s(a(\tau),1,0)$.
%	The curves $s\mapsto G(s,\tau)$ are horizontal, because $\frac{\de a}{\de s}(s,\tau) = a(\tau) X|_{G(s,\tau)} + Y|_{G(s,\tau)}$.
%	So, 
%	 the horizontal normal to $\Gamma_f$ at $G(s,\tau)$ is exactly $\nu(G(s,\tau)) = \frac{a(\tau) X|_{G(s,\tau)} + Y|_{G(s,\tau)}}{\sqrt{1+a(\tau)^2}} $.
	
	By a calibration argument \cite[Theorem~2.3]{MR2333095}, we conclude  that the subgraph $G_f$ is a local perimeter minimizer in $\Omega$.
%	If $a$ is continuous non-decreasing, then there is a sequence $\{a_n\}_{n\in\N}$ of smooth non-decreasing functions so that $a_n\to a$ uniformly on compact sets.
%	It follows that the corresponding sub-graphs $G_{a_n}$ converge to $G_a$ with respect to the Hausdorff distance on compact sets................
\end{proof}

%%%%%%%%%%%%%%%%%%%%%%%%%%%%%%%%%%%%%%%%%%%%%%%%%%%%%%%%%%%%%%%
%%%%%%%%%%%%%%%%%%%%%%%%%%%%%%%%%%%%%%%%%%%%%%%%%%%%%%%%%%%%%%%
\section{First example}\label{sec11241414}

\begin{theorem}\label{thm07111720}
	Define $f:\R^2\to\R$ as
	\[
	f(y,t):=
	\begin{cases}
	 	0 & t\le 0 \\
		\frac{2t}{y} & 0<t\le \frac{y^2}{2}\\
		y & t>\frac{y^2}{2} .
	\end{cases} 
	\]
	Then the following holds:
	\begin{enumerate}[leftmargin=*,label=(\roman*)]
	\item
	$f\in W^{1,p}_{loc}(\R^2)\cap C^0(\R^2)\cap\Lip_{loc}(\R^2\setminus\{0\})$, where $1\le p<3$.
	\item
	$\grad^ff\in\Co^0(\R^2\setminus\{0\})\cap \LI(\R^2)$.
	\item
	$f$ is stable, but $\Gamma_f$ is not an intrinsic plane.
	\item
	The intrinsic graph of $f$ is $\Gamma_f = \Gamma_1\cup\Gamma_2\cup\Gamma_3\subset\HH$, where
	\begin{align*}
 	\Gamma_1 &= \{(0,y,t):\ t\le0,\ y\in\R\} ,\\
	\Gamma_2 &= \{(x,y,0):0\le x\le y\}\cup\{(x,y,0):y\le x\le 0\} ,\\
	\Gamma_3 &= \{(x,y,t):x=y,t\ge 0\} .
	\end{align*}
	The surface $\Gamma_f$ is a cone with respect to the dilations $\delta_\lambda(x,y,z)=(\lambda x,\lambda y,\lambda^2 z)$. 
	\item For each $p\in \Gamma_f\setminus\{0\}$, there is a neighborhood $U$ of p in $\HH$, such that  $\Gamma_f$ is  area minimizing in $U$.	
	\end{enumerate}
\end{theorem}
See Figure~\ref{fig09030019} for an image of the surface $\Gamma_f$.
\begin{figure}
\includegraphics[width=0.5\textwidth]{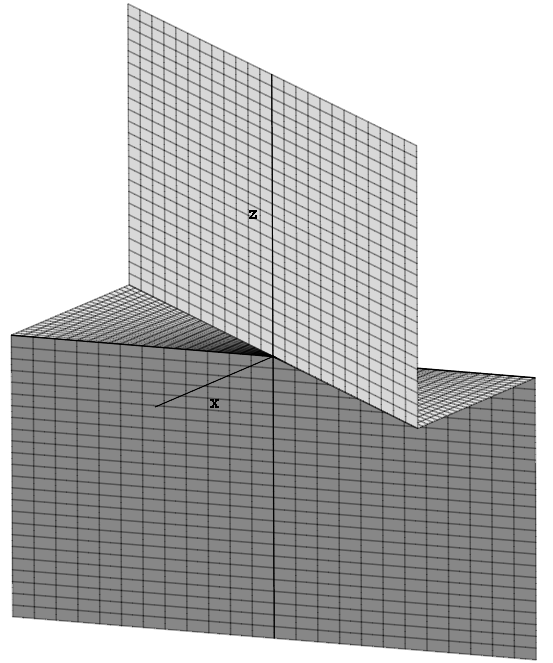}
\caption{\label{fig09030019}
	Image of the surface $\Gamma_f$ from Theorem~\ref{thm07111720}}
\end{figure}
\begin{remark}
	We are not able to prove nor disprove that $\Gamma_f$ is area minimizing in a neighborhood of $(0,0,0)$.
\end{remark}
\begin{remark}\label{rem09042024}
	In a neighborhood of $(1,0)\in\R^2$, the function $f$ above is Lipschitz but not $C^1$.
	Therefore, by \cite[Theorem 1.3]{MR2583494}, $f$ is not a vanishing viscosity solution of the minimal surface equation in the sense of \cite[Definition 1.1]{MR2583494}.
	However, $f$ is a distributional solution to the equation.
\end{remark}
\begin{proof}[Proof of Theorem~\ref{thm07111720}] Point $(iv)$ is immediate. Let us prove point $(v)$. Notice that the vector field
$\nu$ defined in ~\eqref{eq07121716} is a calibration of $\Gamma_f$ in the open half-spaces $S_1:=\{y>\,0\}$ and $S_2:=\{y<\,0\}$, while $X$ is a calibration %of $\Gamma_f$ 
in the open half-space $S_3:=\{t<\,0\}$ and $\frac{X-Y}{\sqrt 2}$ is a calibration in $S_4:=\{t>\,0\}$.
Since every point in $\Gamma_f\setminus\{0\}$ belongs to one of these four open sets, $\Gamma_f\setminus\{0\}$ is locally area minimizing.
%	The fact that the intrinsic graph $\Gamma_f$ of $f$ is locally area minimizing in $\HH\setminus\{0\}$ follows from the fact that $\Gamma_f$ is the union of area minimizing surfaces and that the vector field $\nu$ defined in~\eqref{eq07121716} is a calibration of $\Gamma_f$ on $\bb H\setminus\{x=y=0\}$, as in the proof of Proposition~\ref{prop06211845}.
	This shows $(v)$.
	
	Since $f$ is absolutely continuous along almost every line parallel to the coordinate axes, its distributional derivatives correspond to the pointwise derivatives:
	\begin{equation*}\label{eq06211553}
	 	\de_y f(y,t) = 
		\begin{cases}
		 	0 & t\le 0 \\
		-\frac{2t}{y^2} & 0<t\le \frac{y^2}{2}\\
		1 & t>\frac{y^2}{2} ,
		\end{cases}
	\end{equation*}
	\begin{equation*}\label{eq06211554}
		\de_t f(y,t) =
		\begin{cases}
	 	0 & t\le 0 \\
		\frac{2}{y} & 0<t\le \frac{y^2}{2}\\
		0 & t>\frac{y^2}{2} ,
		\end{cases}
	\end{equation*}	
	\begin{equation*}\label{eq06211459}
	 	\grad^ff(y,t) = (\de_yf+f\de_tf)(y,t) =
		\begin{cases}
	 	0 & t\le 0 \\
		\frac{2t}{y^2} & 0<t\le \frac{y^2}{2}\\
		1 & t>\frac{y^2}{2} .
		\end{cases}
	\end{equation*}
	It is then immediate to see that the parts $(i)$ and $(ii)$ of the theorem are true.
	
	Part $(iii)$ follows from the Lemma~\ref{lem07121720} below.
\end{proof}

\begin{lemma}[Approximation]\label{lem07121556}
	For $\epsilon>0$, define $f_\epsilon:\R^2\to\R$ as
	\[
	f_\epsilon(y,t) :=
	\begin{cases}
	 	0 & t\le 0 \\
		\frac{2yt}{y^2+2\epsilon} & 0<t\le \frac{y^2+2\epsilon}{2}\\
		y & t>\frac{y^2+2\epsilon}{2}
	\end{cases}
	\]
	Then, the following holds for every $\epsilon>0$:
	\begin{enumerate}[label=(\alph*)]
	\item
	$f_\epsilon\in\Lip_{loc}(\R^2)$ and its biLipschitz Lagrangian homeomorphism $\Psi_\epsilon:\R^2\to\R^2$ is $\Psi(s,\tau) = (s,\chi_\epsilon(s,\tau))$ with $\chi_\epsilon(s,\tau)=\frac{a_\epsilon(\tau)}{2} s^2 + \tau$, where
	\[
	a_\epsilon(\tau) := 
	\begin{cases}
	 	0 & \tau\le0 \\
		\frac\tau\epsilon & 0\le\tau\le \epsilon \\
		1 & \epsilon\le\tau .
	\end{cases}
	\]
	\item
	$\lim_{\epsilon\to0^+}f_\epsilon = f$ in $W^{1,p}_{loc}(\R^2)$ for all $p\in[1,3)$.
	\item
	$\grad^{f_\epsilon}f_\epsilon\in C^0(\R^2)$ and
	$\lim_{\epsilon\to0^+}\grad^{f_\epsilon}f_\epsilon = \grad^ff$ in $\LL^p_{loc}(\R^2)$ for all $p\in[1,\infty)$.
	\end{enumerate}
\end{lemma}
\begin{proof}
	Since $f_\epsilon$ is absolutely continuous along almost every line parallel to the coordinate axes, its distributional derivatives are
	\[
	\de_y f_\epsilon = 
	\begin{cases}
	 	0 & t\le 0 \\
		-\frac{2t(y^2-2\epsilon)}{(y^2+2\epsilon)^2} & 0<t\le \frac{y^2+2\epsilon}{2}\\
		1 & t>\frac{y^2+2\epsilon}{2}
	\end{cases}
	\quad\text{and}\quad
	\de_t f_\epsilon =
	\begin{cases}
	 	0 & t\le 0 \\
		\frac{2y}{y^2+2\epsilon} & 0<t\le \frac{y^2+2\epsilon}{2}\\
		0 & t>\frac{y^2+2\epsilon}{2}
	\end{cases}
	.
	\]
	Since both $\de_y f_\epsilon$ and $\de_t f_\epsilon$ are bounded on bounded subsets of $\R^2$, we obtain that $f_\epsilon\in\Lip_{loc}(\R^2)$.
	A direct computation shows that $f(s,\chi_\epsilon(s,\tau)) = \de_s\chi_\epsilon(s,\tau)$ for all $s,\tau\in\R$ and that $\Psi_\epsilon$ is indeed biLipschitz.
	So, part $(a)$ holds.
	
	Let us now observe that $f_\epsilon\to f$, $\de_yf_\epsilon\to \de_yf$ and $\de_tf_\epsilon\to \de_tf$ pointwise almost everywhere in $\R^2$.
	Moreover, $|f_\epsilon|\le g_1$, $|\de_yf_\epsilon|\le g_2$ and $|\de_tf_\epsilon|\le g_3$ almost everywhere in $\R^2$, where 
	\begin{align*}
	 	g_1(y,t) &:= |y|, \qquad
		g_2(y,t) := 1, \\
		g_3(y,t) &:= 
			\begin{cases}
		 	0 & t\le 0 \\
			\frac{2}{|y|} & 0<t\le \frac{y^2}{2}\\
			\frac{\sqrt2}{\sqrt t} & \frac{y^2}{2}<t<\frac{y^2}{2} +1 \\
			0 & \frac{y^2}{2}+1<t .
			\end{cases}
	\end{align*}
	Since, for every $L>0$ and $p\neq2$,
	\[
	\int_{[-L,L]^2} \!\!\!\!|g_3(y,t)|^p \dd y\dd t
	\!=\!\! \left(1-p\frac{2^{p-1}}{2-p}\right) \!\!\int_{-L}^L \!\!|y|^{2-p} \dd y 
		+ \frac{2^p}{2-p} \!\!\int_{-L}^{L} \!\!\left(y^2+2\right)^{\frac{2-p}{2}} \!\! \dd y ,
	\]
	then $g_3\in\LL^p_{loc}(\R^2)$ for all $1\le p<3$.
	Clearly, we also have $g_1,g_2\in\LL^p_{loc}(\R^2)$ for all $1\le p<3$.
	Therefore, by the Dominated Convergence Theorem, $f_\epsilon\to f$ in $W^{1,p}_{loc}(\R^2)$ for all $1\le p<3$, i.e., statement $(b)$ in the lemma.
	
	For part $(c)$, one can check by direct computation that
	\[
	\grad^{f_\epsilon}f_\epsilon = 
	\begin{cases}
	 	0 & t\le 0 \\
		\frac{2 t}{y^2+2\epsilon} & 0<t\le \frac{y^2+2\epsilon}{2}\\
		1 & t>\frac{y^2+2\epsilon}{2}
	\end{cases}
	\]
	Moreover, we have $\grad^{f_\epsilon}f_\epsilon \to \grad^ff$ in $\LL^p_{loc}(\R^2)$ for all $p\in[1,\infty)$.
	Indeed, on one hand the pointwise convergence $\grad^{f_\epsilon}f_\epsilon$ in $\R^2$ is clear.
	On the other hand, $|\grad^{f_\epsilon}f_\epsilon(y,t)|\le 1$ for a.e.~$(y,t)\in\R^2$ and for all $\epsilon\in(0,1)$, and therefore we can conclude again by the Dominated Convergence Theorem.
\end{proof}

\begin{lemma}[Stability]\label{lem07121720}
	The function $f$ defined in Theorem~\ref{thm07111720} is stable.
\end{lemma}
\begin{proof}[Proof for ``$f$ satisfies \eqref{eq06221238}'']
	Let $f_\epsilon$ as in Lemma~\ref{lem07121556} and $\varphi\in C^\infty_c(\R^2)$.
	Since $f_\epsilon\to f$ in $W^{1,2}_{loc}(\R^2)$, then $I_f(\varphi)=\lim_{\epsilon\to0}I_{f_\epsilon}(\varphi)$ by Lemma~\ref{lem08261641}.
	By Proposition~\ref{prop06061521}.(iii), $I_{f_\epsilon}(\varphi)=0$ for all $\epsilon$, thus $I_f(\varphi)=0$.
\end{proof}

\begin{proof}[Proof for ``$f$ satisfies \eqref{eq06221239b}'']
	Let $f_\epsilon$ and $a_\epsilon$ as in Lemma~\ref{lem07121556} and $\varphi\in C^\infty_c(\R^2)$.
	Since $f_\epsilon\to f$ in $W^{1,2}_{loc}(\R^2)$, then $II_f(\varphi)=\lim_{\epsilon\to0}II_{f_\epsilon}(\varphi)$ by Lemma~\ref{lem08261641}.
	By Proposition~\ref{prop06061521}.(iv), 
	\[
	II_{f_\epsilon}(\varphi) 
	= \int_{\R^2}
		\left[(\de_s\tilde\varphi_\epsilon)^2 \frac{(\frac{a'_\epsilon}2s^2+1)}{(1+a_\epsilon^2)^{3/2}} 
		-2 \tilde\varphi_\epsilon^2 \frac{ a_\epsilon' }{ (1+a_\epsilon^2)^{3/2} (\frac{a_\epsilon'}2s^2+1) }
	\right]\dd s\dd\tau ,
	\]
	where we have $\tilde\varphi_\epsilon(s,\tau)=\varphi(s,\chi_\epsilon(s,\tau))$, $\chi_\epsilon(s,\tau)=\frac{a_\epsilon(\tau)}{2} s^2+\tau$ and $\tilde f_\epsilon(s,\tau)=\de_s\chi_\epsilon(s,\tau)=a_\epsilon(\tau)s$.
	Since 
	\[
	(\de_s\tilde\varphi_\epsilon)^2\frac{(\frac{a_\epsilon'}2s^2+1)}{(1+a_\epsilon^2)^{3/2}} \ge 0,
	\]
	the thesis follows if it is true that
	\begin{equation}\label{eq06241642}
	 	\limsup_{\epsilon\to0}
		\int_{\R^2} 
		\tilde\varphi_\epsilon^2
		\frac{
			 a_\epsilon'
		}{
			(1+(\de_s\tilde f_\epsilon)^2)^{3/2} (\frac{a_\epsilon'}2s^2+1)
		}
		\dd s\dd\tau
		\le0 .
	\end{equation}
	For proving \eqref{eq06241642}, 
	Recall that $a_\epsilon'(\tau)=1/\epsilon$ for $\tau\in[0,\epsilon]$ and $0$ otherwise.
	So, if we perform the change of variables $v=\frac{s}{\sqrt{2\epsilon}}$ and $w=\frac\tau\epsilon$, we obtain
	\begin{multline*}
	 	\int_{\R^2} 
		\tilde\varphi_\epsilon^2
		\frac{
			 a_\epsilon'
		}{
			(1+(a_\epsilon)^2)^{3/2} (\frac{a_\epsilon'}2s^2+1)
		}
		\dd\LL^2(s,\tau) \\
		=
		\int_{\R} \int_0^\epsilon
		\varphi(s,\frac{\tau s^2}{2\epsilon}+\tau)^2
		\frac{
			1/\epsilon
		}{
			(1+(\tau/\epsilon)^2)^{3/2} (\frac1{2\epsilon}s^2+1)
		}
		\dd \tau\dd s \\
		=
		\sqrt{2\epsilon}\int_{\R} \int_0^1
		\varphi(\sqrt{2\epsilon}v,\epsilon w(v^2+1))^2
		\frac{
			1
		}{
			(1+w^2)^{3/2} (v^2+1)
		}
		\dd w\dd v \\
		\le 
		\sqrt{2\epsilon} M 
		\int_{\R}\frac1{v^2+1}\dd v
		\int_0^1 \frac1{(1+w^2)^{3/2}}\dd w
	\end{multline*}
	where $M=\sup_{\R^2}\varphi^2$.
	Since $\int_{\R}\frac1{v^2+1}\dd v\int_0^1 \frac1{(1+w^2)^{3/2}}\dd w<\infty$, taking the limsup as $\epsilon\to0$ we get \eqref{eq06241642}.
\end{proof}

%%%%%%%%%%%%%%%%%%%%%%%%%%%%%%%%%%%%%%%%%%%%%%%%%%%%%%%%%%%%%%%
%%%%%%%%%%%%%%%%%%%%%%%%%%%%%%%%%%%%%%%%%%%%%%%%%%%%%%%%%%%%%%%
%\newpage
\section{Second example}\label{sec09011226}
In this section we construct the example that proves Theorem~\ref{thm06102321}.
We summarize the results in the following statement, whose proof covers the whole section.
A plot of the graph $\Gamma_f$ can be found in Figure~\ref{fig09030022}.
\begin{figure}
\includegraphics[width=0.5\textwidth]{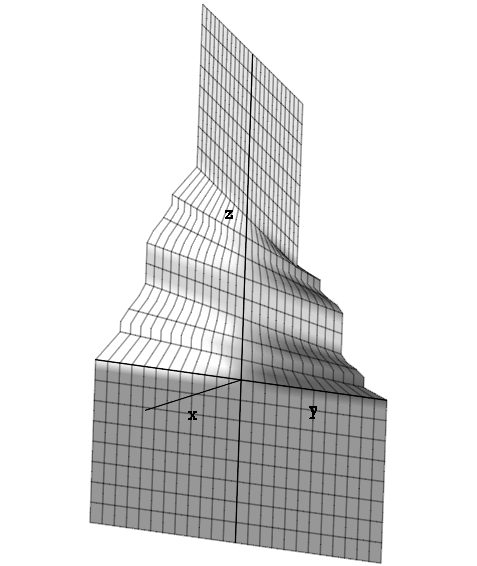}
\caption{\label{fig09030022}
	Image of the surface $\Gamma_f$ from Theorem~\ref{thm09011227}}
\end{figure}

\begin{theorem}\label{thm09011227}
	Let $a:\R\to[0,1]$ be the function that is the Cantor staircase when restricted to $[0,1]$ and with $a(\tau)=0$ for $\tau\le0$, $a(\tau)=1$ for $\tau\ge1$.
	Let $f\in\Cow(\R^2)$ be the function such that $f(s,a(\tau)s^2/2+\tau) = a(\tau)s$, as in Proposition~\ref{prop06061521}.	
	Then the following holds:
	\begin{enumerate}[leftmargin=*,label=(\roman*)]
	\item
	$f\in W^{1,2}_{loc}(\R^2)\cap\Cow(\R^2)\cap\Lip_{loc}(\R^2\setminus(\{0\}\times\R))$.
	\item
	$f$ is stable, but $\Gamma_f$ is not an intrinsic plane.
	The surface $\Gamma_f$ is locally area minimizing in $\HH\setminus\{(0,0,z):z\in C\}$, where $C\subset[0,1]$ is the ternary Cantor set.	
	\end{enumerate}
\end{theorem}

The fact that $\Gamma_f$ is not an intrinsic plane is clear.
The fact that $\Gamma_f$ is locally area minimizing in $\HH\setminus\{(0,0,z):z\in C\}$
is proven as in Theorem~\ref{thm07111720}.$(v)$: More precisely, if $p\in\Gamma_f\setminus\{x=y=0\}$, then $\nu$ is a local calibration by Proposition~\ref{prop06211845}; 
if $p\in\{(0,0,z):z\notin C\}$, then there is $U\subset\HH$ open, $p\in U$ so that $U\cap\Gamma_f$ is a subset of an intrinsic plane.
% because of Proposition~\ref{prop06211845} and the fact that both $\Gamma_f\cap\{(x,y,z):z>1\}$ and $\Gamma_f\cap\{(x,y,z):z<0\}$ are vertical half-planes, which are local area minimizers.

The fact that $f\in \Cow(\R^2)\cap\Lip_{loc}(\R^2\setminus(\{0\}\times\R))$ follows from Proposition~\ref{prop06061521}.
For proving that $f$ is stable, 
we shall construct a Lipschitz approximation of $a$ and then complete the proof by approximation.
In particular, we show through Lemma~\ref{lem08270151} that $f\in W^{1,2}_{loc}(\R^2)$.
Finally we will estimate the first and the second variations of $f$ in Lemmas~\ref{lem08270101} and~\ref{lem08270102}.
\\

Define the closed sets $C(n)\subset[0,1]$, $n\in\N$, inductively as follows:
$C(0):=[0,1]$ and 
\[
C(n+1) := \frac13 C(n) \cup \left(\frac23 + \frac13 C(n)\right) .
\]
For $k,n\in\N$, define
%let $C(n,k)$ be the connected component of $C(n)$ of the form $[k/3^n,(k+1)/3^n]$, i.e.,
\[
C(n,k) :=
\begin{cases}
\left[\frac{k}{3^n},\frac{k+1}{3^n}\right] &\text{if }[k/3^n,(k+1)/3^n]\subset C(n) \\
\emptyset &\text{otherwise}.
\end{cases}
\]

Let $J_n$ be the collection of $k\in\{0,\dots,3^n\}$ such that $C(n,k)\neq\emptyset$.
We have $\# J_n=2^n$ and $C(n)=\bigsqcup_{k\in J_n} C(n,k)$.
Moreover, 
\[
C=\bigcap_{n=1}^\infty C(n)
\]
is the ternary Cantor set in $[0,1]$.
Set $q := \frac23$.

For $n\in\N$, let $a_n:\,[0,1]\to [0,1]$ be tha classical sequence of piecewise affine functions for which $a_n\to a$ uniformly on $[0,1]$ and $a$ agrees wiht  the Cantor staircase function. A possible way for defining $(a_n)_n$ is the following one.
For $n\in\N$, define $a_n:\R\to[0,1]$ as the absolutely continuous function $a_n(\tau)=\int_{-\infty}^\tau a_n'(r)\dd r$, where $a_n'(r) :=\frac1{q^n} \one_{C(n)}(r)$.
Then $a_n\to a$ uniformly on $\R$, where $a:\R\to\R$ is the function such that $a(\tau)=0$ for $\tau\le0$, $a(\tau)=1$ for $\tau\ge1$ and $a|_{[0,1]}$ is the Cantor function on the ternary Cantor set $C$.
Notice that $a(\tau)=a_n(\tau)$ for all $\tau\in\R\setminus C(n)$.
By continuity, the equality holds also on $\de C(n)$.

For $ y\in\R$ and $k\in J_n$ define the following subsets of $\R$:
\begin{align*}
C_ y(n,k) &:= \left\{ a_n(\tau) \frac{ y^2}2 + \tau : \tau\in C(n,k) \right\} \\
&= \left[ a\left(\frac{k}{3^n}\right) \frac{ y^2}2 + \frac{k}{3^n} , a\left(\frac{k+1}{3^n}\right) \frac{ y^2}2 + \frac{k+1}{3^n} \right] ;\\
C_ y(n) &:= \left\{a_n(\tau)\frac{ y^2}2 + \tau : \tau\in C(n) \right\} 
	= \bigsqcup_{k\in J_n} C_ y(n,k) ; \\
C_ y &:= \left\{a(\tau)\frac{ y^2}2 + \tau : \tau\in C \right\} 
	= \bigcap_{n=1}^\infty C_ y(n) .
\end{align*}

Notice that
\begin{align*}
	\LL^1(C_ y(n,k) &= \frac1{2^n} \left(\frac{ y^2}{2} + q^n\right), &
	\LL^1(C_ y(n)) &= \frac{ y^2}{2} + q^n , &
	\LL^1(C_ y) &= \frac{ y^2}{2} . 
\end{align*}

For each $n\in\N$, define $f_n\in\Cow$ as the function such that, for all $(s,\tau)\in\R^2$, $f_n\left(s,a_n(\tau)s^2/2 +\tau \right) = a_n(\tau) s$, as in Proposition~\ref{prop06061521}.
Since $a_n$ is locally Lipschitz, $f_n$ is locally Lipschitz as well, for all $n$.

\begin{lemma}[Approximation]\label{lem08270151}
	The sequence of functions $f_n$ defined above converge to $f$ in $W^{1,2}_{loc}(\R^2)$.
	In particular, $f\in W^{1,2}_{loc}(\R^2)$.
\end{lemma}
\begin{proof}
	First of all, we claim that $f$ is absolutely continuous along almost all coordinates lines.
	Indeed, by Proposition~\ref{prop06061521}, $f$ is locally Lipschitz on $\R^2\setminus\{ y=0\}$ and thus $ t\mapsto f( y, t)$ is absolutely continuous if $ y\neq0$.
	Moreover, if $ t\notin C$, then $ y\mapsto f( y, t)$ is constant in a neighborhood of $0$, so it is absolutely continuous on $\R$.
	Since $\LL^1(C)=0$, this completes the proof of the claim.
	
	Therefore, the distributional derivatives $\de_ y f$ and $\de_ t f$ are functions and coincide almost everywhere with the derivatives of $f$ along the coordinates lines.
	
	We compute
	\begin{equation}\label{eq08270128}
	\de_ t f_n( y, t) =
	\begin{cases}
	\frac{ y}{ y^2/2 + q^n } &\text{if }  t\in C_ y(n) \\
	0 &\text{if }  t\notin C_ y(n) .
	\end{cases}
	\end{equation}
	Since $a_n$ is piecewise affine, then $ t\mapsto f_n( y, t)$ is also piecewise affine.
	If $ t\notin C_ y(n)$, then $\de_ t f_n( y, t)=0$.
	If $k\in J_n$, then $ t\mapsto\de_ t f_n( y, t)$ is constant on $C_ y(n,k) = [ t_1, t_2]$.
	Thus
	\begin{align*}
	\de_ t f_n( y, t)
	&= \frac{ f_n( y, t_2)-f_n( y, t_1) }{  t_2- t_1 }
	= \frac{ a(\frac{k+1}{3^n}) y - a(\frac{k}{3^n}) y }{ a(\frac{k+1}{3^n}) y^2/2 + \frac{k+1}{3^n} - a(\frac{k}{3^n}) y^2/2 - \frac{k}{3^n} }\\
	&= \frac{ y}{ y^2/2 + q^n }
	\end{align*}
	This shows~\eqref{eq08270128}.
	Next, we show that 
	\begin{equation}\label{eq08270129}
	\de_ t f( y, t) =
	\begin{cases}
	\frac2 y &\text{for a.e.~} t\in C_ y \\
	0 &\text{otherwise} .
	\end{cases}
	\end{equation}
	Fix $ y\in\R$.
	So, if $ t\notin C_ y$, then $ t'\mapsto f( y, t')$ is constant in a neighborhood of $ t$, hence $\de_ t f( y, t)=0$.
	If $ y=0$, then $C_0=C$ has measure zero.
	Let $ y\neq0$ and $ t\in C_ y$ be such that $ t'\mapsto f( y, t')$ is differentiable at $ t$.
	Then, if $ t=a(\tau)\frac{ y^2}{2}+\tau$ and $ t'=a(\tau')\frac{ y^2}{2}+\tau'$, we have
	\begin{multline*}
	\limsup_{ t'\to t} \frac{|f( y, t')-f( y, t)|}{| t'- t|}
	= \limsup_{\tau'\to\tau} \frac{ (a(\tau')-a(\tau))  y }{ (a(\tau')-a(\tau) ) \frac { y^2}2 + (\tau'-\tau) } \\
	\le \frac2 y \limsup_{\tau'\to\tau} \frac{1}{1+\frac2{ y^2} \frac{\tau'-\tau}{a(\tau')-a(\tau)}}
	\le \frac2 y .
	\end{multline*}
	Moreover, if $ y>0$ is such that $ t\mapsto f( y, t)$ is absolutely continuous, which happens for almost every $ y\in\R$ by Proposition ~\ref{prop06061521}, from the inequalities
	\[
	 y 
	= f( y,\frac12 y^2+1)
	= \int_0^{\frac12 y^2+1} \de_ t f( y, t) \dd t
	= \int_{C_ y} \de_ t f( y, t) \dd t
	\le \frac2{ y} |C_ y|
	=  y
	\]
	follows that $\de_ t f( y, t)= \frac2 y$.
	The same strategy applies to the case $ y<0$ and so we have~\eqref{eq08270129}.
	
	Now we prove the first convergence, that is,
	\begin{equation}\label{eq08270132}
	\de_ t f_n\to\de_ t f \text{ in }\LL^2_{loc}(\R^2) .
	\end{equation}
	We directly compute
	\begin{align*}
	\int_0^\ell &\int_\R |\de_ t f_n( y, t)-\de_ t f( y, t)|^2 \dd t \dd y \\
	&= \int_0^\ell \left[
		\int_{C_ y(n)\setminus C_ y} \left(\frac{ y}{ y^2/2+q^n}\right)^2 \dd t
		+ \int_{C_ y} \left|\frac{ y}{ y^2/2+q^n}-\frac2 y\right|^2 \dd t
		\right]\dd y \\
	&= \int_0^\ell \left[
		\left(\frac{ y}{ y^2/2+q^n}\right)^2 \left(|C_ y(n)|-|C_ y|\right)
		+  \left|\frac{ y}{ y^2/2+q^n}-\frac2 y\right|^2 |C_ y|
		\right]\dd y \\
	&= q^n \int_0^\ell \left(\frac{ y}{ y^2/2+q^n}\right)^2  \dd y
		+ 2 q^{2n} \int_0^\ell \frac{1}{( y^2/2+q^n)^2} \dd y \\
	&= 4q^n \int_0^\ell \frac{1}{ y^2/2+q^n} \dd y \\
	&= 2 \int_0^{\ell/\sqrt{2q^n}} \frac{\sqrt{2q^n}}{x^2+1} \dd x
	= 2\sqrt{2} q^{n/2} \arctan\left(\frac{\ell}{\sqrt2} q^{-n/2}\right) ,
	\end{align*}
	The last expression goes to $0$ as $n\to\infty$, and so~\eqref{eq08270132} is proven.
	
	The next step is to show that 
	\begin{equation}\label{eq08270134}
	f_n\to f\text{ and }\grad^{f_n}f_n\to \grad^ff\text{ uniformly on compact sets}.
	\end{equation}
	First of all, if $\epsilon>0$ and $K>0$, then there is $N>0$ such that for all $ t\in\R$, all $ y\in\R$ with $| y|\le K$ and all $n\ge N$ there are $\tau_1,\tau_2\in\R\setminus C(n)$ such that 
	\[
	 t_1:=a(\tau_1)\frac{ y^2}2+\tau_1 
	\le  t \le 
	a(\tau_2)\frac{ y^2}2+\tau_2=: t_2,
	\] 
	and $a(\tau_2)-a(\tau_1) \le \epsilon$.
	
	Secondly, notice that $a_n=a$ on $\R\setminus C(n)$.
	So,
	\begin{align*}
	&\grad^{f_n}f_n( y, t) - \grad^ff( y, t) 
	\le \grad^{f_n}f_n( y, t_2) - \grad^ff( y, t_1)
	= a(\tau_2) - a(\tau_1)
	\le \epsilon , \\
	&\grad^ff( y, t) - \grad^{f_n}f_n( y, t) 
	\le \grad^ff( y, t_2) - \grad^{f_n}f_n( y, t_1)
	= a(\tau_2) - a(\tau_1)
	\le \epsilon .
	\end{align*}
	Therefore, there is $N\in\N$ such that for all $( y, t)\in\R^2$ with $| y|\le K$ and all $n\ge N$, $|\grad^{f_n}f_n( y, t) - \grad^ff( y, t)|\le \epsilon$, i.e., $\grad^{f_n}f_n\to \grad^ff$ uniformly on compact sets.
	
	Next, notice that $f( y, t)=\grad^ff( y, t)\, y$ and $f_n( y, t)=\grad^{f_n}f_n( y, t)\, y$.
	Therefore, $f_n\to f$ uniformly on compact sets as well and \eqref{eq08270134} is proven.
	
	Finally, we conclude that
	\begin{equation}
	\de_ y f_n\to\de_ y f\text{ in }\LL^2_{loc}(\R^2) .
	\end{equation}
	Indeed, since $f$ is ACL, we have, whenever $\grad f_n$ exist for all $n$,
	\[
	\de_ y f_n = \grad^{f_n}f_n - f_n\de_ t f_n .
	\]
	Since the right hand side converges to $\de_ y f$ in $\LL^2_{loc}(\R^2)$, the left hand side does the same.
	The proof is complete.
\end{proof}

\begin{lemma}\label{lem08270101}
	The function $f$ satisfies \eqref{eq06221238}.
\end{lemma}
\begin{proof}
	Let $f_n$ as above and $\varphi\in C^\infty_c(\R^2)$.
	Since $f_n\to f$ in $W^{1,2}_{loc}(\R^2)$, then $I_f(\varphi)=\lim_{n\to\infty}I_{f_n}(\varphi)$ by Lemma~\ref{lem08261641}.
	By Proposition~\ref{prop06061521}.(iii), $I_{f_n}(\varphi)=0$ for all $n$, thus $I_f(\varphi)=0$.
\end{proof}

\begin{lemma}\label{lem08270102}
	The function $f$ satisfies \eqref{eq06221239b}.
\end{lemma}
\begin{proof}
	Fix $\varphi\in\Co^\infty_c(\R^2)$.
	Since $f_n\to f$ in $W^{1,2}_{loc}(\R^2)$ and $\grad^{f_n}f_n\to\grad^ff$ uniformly on compact sets, then 
	\[
	II_f(\varphi) = \lim_{n\to\infty} II_{f_n}(\varphi) .
	\]
	Since $a_n$ is locally Lipschitz, by Proposition~\ref{prop06061521}.$(iv)$, we have
	\[
	II_{f_n}(\varphi) 
	= \int_{\R^2}
		(\de_s\tilde\varphi_n)^2 \frac{(\frac{a_n'}2s^2+1)}{(1+a_n^2)^{3/2}} 
		-2 \tilde\varphi_n^2 \frac{ a_n' }{ (1+(a_n)^2)^{3/2} (\frac{a_n'}2s^2+1) }
	\dd s\dd\tau ,
	\]
	where $\tilde\varphi_n(s,\tau) := \varphi(s,a_n(\tau)\frac{s^2}{2}+\tau)$.
	So, we only need to show that
	\[
	\limsup_{n\to\infty} 
	\int_{\R^2}
		 \tilde\varphi_n^2 \frac{ a_n' }{ (1+(a_n)^2)^{3/2} (\frac{a_n'}2t^2+1) }
	\dd t\dd\tau 
	\le 0 .
	\]
	Let $M:=\sup_{p\in\R^2}\varphi(p)^2$.
	Then
	\begin{multline*}
	\int_{\R^2}
		 \tilde\varphi_n^2 \frac{ a_n' }{ (1+(a_n)^2)^{3/2} (\frac{a_n'}2s^2+1) }
	\dd t\dd\tau \\
	\le 
	M \int_{C(n)}
		 \frac{ a_n' }{ (1+(a_n)^2)^{3/2} } \int_\R\frac{1}{ (\frac{a_n'}{2}s^2+1) } \dd t
	\dd\tau
	\end{multline*}
	If $\tau\in C(n)$, then $a_n'=q^{-n}$ and, after substituting $v=\sqrt{\frac{1}{2q^n}}s$, $\dd v = \sqrt{\frac{1}{2q^n}}\dd s$
	\[
	\int_\R\frac{1}{ (\frac{a_n'}{2}s^2+1) } \dd s
	= \int_\R \frac{\sqrt{2q^n}}{v^2+1} \dd v
	= \sqrt{2q^n} \pi .
	\]
	Moreover,
	\[
	\int_{C(n)} \frac{ a_n' }{ (1+(a_n)^2)^{3/2} } \dd \tau
	=  \sum_{k\in J_k} \int_{\frac{k}{3^n}}^{\frac{k+1}{3^n}} \frac{a_n'}{ (1+(a_n)^2)^{3/2} } \dd \tau
	\]
	For each $k\in J_k$, make the substitution $v=a_n(\tau)$, $\dd v=a_n'\dd \tau$, $\frac{k}{3^n}\mapsto a_n(\frac{k}{3^n})=a(\frac{k}{3^n})$, $\frac{k+1}{3^n}\mapsto a_n(\frac{k+1}{3^n})=a(\frac{k+1}{3^n})$
	\[
	\int_{\frac{k}{3^n}}^{\frac{k+1}{3^n}} \frac{a_n'}{ (1+(a_n)^2)^{3/2} } \dd \tau
	= \int_{a(\frac{k}{3^n})}^{a(\frac{k+1}{3^n})} \frac{1}{ (1+v^2)^{3/2} } \dd v
	\]
	So,
	\[
	\sum_{k\in J_k} \int_{\frac{k}{3^n}}^{\frac{k+1}{3^n}} \frac{a_n'}{ (1+(a_n)^2)^{3/2} } \dd \tau
	= \int_0^1 \frac{1}{ (1+v^2)^{3/2} } \dd v
	= \frac1{\sqrt2} .
	\]
	All in all, we have
	\[
	\limsup_{n\to\infty} 
	\int_{\R^2}
		 \varphi_n^2 \frac{ a_n' }{ (1+(a_n)^2)^{3/2} (\frac{a_n'}2s^2+1) }
	\dd s\dd\tau
	\le \limsup_{n\to\infty} M\sqrt{q^n} \pi
	= 0 .
	\]
\end{proof}

%\nocite{*}
\printbibliography
\end{document}